\documentclass[11pt]{article}
\usepackage{amsthm}
\usepackage{amsmath}
\usepackage{amscd}

\usepackage[all]{xy}
\usepackage{tikz}
\usetikzlibrary{positioning} 
\usetikzlibrary{arrows,decorations.pathmorphing,shapes} 
\usepackage{tikz-cd}

\usepackage{graphicx, color}
\usepackage[T1]{fontenc}
\usepackage{caption}
\usepackage[latin2]{inputenc}
\usepackage[mathscr]{eucal}
\usepackage{indentfirst}
\usepackage{graphicx}
\usepackage{graphics}
\usepackage{pict2e}
\usepackage{epic}
\numberwithin{equation}{section}
\usepackage[margin=2.9cm]{geometry}
\usepackage{epstopdf} 
\usepackage{latexsym}
\usepackage[hyperfootnotes=false]{hyperref}
\usepackage{xargs}  

\usepackage{titlefoot}

\usepackage{float}
\usepackage{amsmath, amsfonts, amssymb, amsxtra, amsthm, mathrsfs}
\usepackage{enumitem}

\usepackage[colorinlistoftodos,prependcaption,textsize=tiny]{todonotes}

\newcommandx{\unsure}[2][1=]{\todo[linecolor=red,backgroundcolor=red!25,bordercolor=red,#1]{#2}}
\newcommandx{\change}[2][1=]{\todo[linecolor=blue,backgroundcolor=blue!25,bordercolor=blue,#1]{#2}}
\newcommandx{\info}[2][1=]{\todo[linecolor=OliveGreen,backgroundcolor=OliveGreen!25,bordercolor=OliveGreen,#1]{#2}}
\newcommandx{\improvement}[2][1=]{\todo[linecolor=Plum,backgroundcolor=Plum!25,bordercolor=Plum,#1]{#2}}

\usepackage[calcwidth]{titlesec}
\newcommand{\periodafter}[1]{#1.}
\titleformat{\paragraph}[runin]{\bfseries}{\theparagraph}{}{\periodafter}

\theoremstyle{plain}
\newtheorem{Thm}{Theorem}[section]
\newtheorem{Lemma}[Thm]{Lemma}
\newtheorem{Cor}[Thm]{Corollary}
\newtheorem{Prop}[Thm]{Proposition}
\newtheorem{Claim}{Claim}[section]
\newtheorem{ClaimRoman}{Claim}

 \theoremstyle{definition}

\newtheorem{Rmk}[Thm]{Remark}
\newtheorem{?}[Thm]{Problem}
\newtheorem{constr}{Constraint}[section]
\newtheorem{Ex}{Example}
\newtheorem*{Rmk*}{Remark}
\newenvironment{Ex*}
 {\pushQED{\qed}\Ex}
 {\popQED\endEx}

\newcommand{\diam}{{\rm{diam}}}

\newcommand{\HS}{\mathcal{H}_{\epsilon_d}(\star)}
\newcommand{\PH}{\partial \mathcal{H}_{\epsilon_d}(\star)}

\DeclareMathOperator{\Lip}{Lip}
\DeclareMathOperator{\Sob}{Sob}
\DeclareMathOperator{\vol}{vol}
\DeclareMathOperator{\disc}{disc}
\DeclareMathOperator{\SO}{SO}
\DeclareMathOperator{\GL}{GL}

\DeclareMathOperator{\inj}{inj}
\DeclareMathOperator{\Gr}{\mathbf{Gr}}

\DeclareMathOperator{\Conf}{\mathbf{Conf}}
\DeclareMathOperator{\isom}{Isom}
\DeclareMathOperator{\supp}{supp}
\DeclareMathOperator{\spanset}{span}
\DeclareMathOperator{\bigO}{O}

\allowdisplaybreaks

\usepackage{tocloft}
\usepackage{authblk}

\cftpagenumbersoff{part}

\title{\bfseries Effective equidistribution of intersection points in hyperbolic manifolds}
\author{Tina Torkaman and Yongquan Zhang}
\date{November 30, 2025}
\amssubj{53C42, 53C22, 53C30.}

\begin{document}
\maketitle
\begin{abstract}
    In this paper, we establish effective equidistribution of transverse intersection points between properly immersed totally geodesic submanifolds of complementary dimensions in a finite-volume hyperbolic manifold with respect to the hyperbolic volume measure, as the volume of the submanifolds tends to infinity.
\end{abstract}

\tableofcontents

\section{Introduction}
Let $M$ be a complete hyperbolic $d$-manifold, i.e.\ a complete Riemannian $d$-manifold of constant sectional curvature $-1$. In this paper, we always assume that $M$ has finite volume. The main goal of this paper is to study the distribution of transverse intersection points between properly immersed totally geodesic submanifolds of complementary dimensions in $M$. First, we discuss the case of closed geodesics and geodesic hypersurfaces.

\paragraph{Intersection of geodesics and hypersurfaces}
It is well-known that closed geodesics are equidistributed on $M$ when $M$ has finite volume \cite{anosov_syst, geodesic_equidistr, roblin}. When $d>2$, it is also known that any infinite sequence of distinct properly immersed totally geodesic hypersurfaces (which we call \emph{geodesic hypersurfaces} for simplicity) is equidistributed \cite{measure_equidistr}, but such a sequence may not exist. A recent remarkable result states that if $M$ contains infinitely many geodesic hypersurfaces, then it is arithmetic \cite{arithmetic_geodesic} (see also \cite{arithmetic_geodesic_dim3} in dimension 3). In fact, we only need to assume that $M$ is arithmetic and contains at least one geodesic hypersurface. It then follows from the density of the commensurator of an arithmetic group \cite{dense_commensurator} that $M$ contains infinitely many distinct geodesic hypersurfaces (see also \cite[Theorem~9.5.6]{arithmetic} in dimension 3). Such examples can be found in abundance, see e.g.\ \cite[\S2.8]{Non-arithmetic}.

Let $M$ be an arithmetic hyperbolic manifold of dimension $d\ge3$ containing at least one geodesic hypersurface. Given any positive number $\iota$ smaller than the Margulis constant, denote by $M_{\iota}$ the union of $\iota$-thin cuspidal neighborhoods for each cusp of $M$ (in particular, $M_\iota=\varnothing$ when $M$ is \emph{closed}, i.e.\ compact without boundary).
Let $\gamma_T$ be the sum of closed geodesics on $M$ with length $\le T$, and $\{S_n\}$ an infinite sequence of distinct geodesic hypersurfaces in $M$.
We want to determine the distribution of the intersection points between $\gamma_T$ and $S_n$ on $M$ as $n, T \to \infty$, with effective error terms.

Let $I(\gamma_T, S_n)$ denote the sum of the delta measures supported on the transverse intersection points between $\gamma_T$ and the hypersurface $S_n$, each weighted by its multiplicity. Let $\vol$ denote the hyperbolic volume measure on $M$. Given a finite measure $\mu$, let $|\mu|$ denote its total volume and $\mu^1:=\mu/|\mu|$ the normalized probability measure. For simplicity, throughout this paper, we use $C,\epsilon,\epsilon_1,\epsilon_2, \epsilon_2',q$ to denote certain positive constants depending only on the geometry of $M$, but their exact values may change for each appearance. Our main result is the following.

\begin{Thm}\label{thm: main}
Let $M$ be an arithmetic hyperbolic manifold of dimension $d \geq 3$. Let $\gamma_T$ be the sum of closed geodesics with length $\leq T$ and $\{S_n\}$ a sequence of distinct geodesic hypersurfaces in $M$. Then there exist positive constants $C, \epsilon_1,\epsilon_2,q$ depending only on the geometry of $M$, so that for any compactly supported smooth function $f \in C_c^{\infty}(M)$, we have 
$$
\left|\int_{M} f \, d\vol^1-\int_{M} f \, d\,I^1(\gamma_T,S_n)\right| \leq C(e^{-\epsilon_1T}+\vol(S_n)^{-\epsilon_2})\Sob_{q}(f),
$$
where $\Sob_{q}(f)$ is a degree $q$ Sobolev norm of $f$.
\end{Thm}
\begin{Rmk}[On the Sobolev norm]\label{rmk: Sob}
   The Sobolev norm used here is weighted by a function that diverges near the cusps; see the detailed discussion in \cite{effective.surface}. In fact, \cite{effective.surface} shows that one can choose the parameter $q$ such that the norm $\Sob_q(f)$ controls both $\sup |f|$ and the Lipschitz norm $\Lip(f)$, up to a multiplicative constant.
\end{Rmk}

An effective equidistribution result also holds in the case $d=2$. Indeed, since our proof relies on the equidistribution of sequences of geodesic objects, the same method yields the following result.

\begin{Thm}\label{thm: main,d=2}
    Let $M$ be a complete hyperbolic surface of finite area. Let $\gamma_T$ be the sum of the closed geodesics in $M$ with length $\leq T$. Then there exist positive constants $C,\epsilon, q$ so that for any $f \in C^{\infty}_c(M)$, we have
    $$
    \left|\int_{M} f \, d\vol^1-\int_{M} f \, d\,I^1(\gamma_{T_1},\gamma_{T_2})\right| \leq C e^{-\epsilon \min\{T_1,T_2\}} \Sob_{q}(f),
    $$
    where $\Sob_{q}(f)$ is some degree $q$ Sobolev norm of $f$.
\end{Thm}
We remark that for $M$ compact, this is a special case of a recent theorem of Katz in the more general setting of negatively curved surfaces \cite{katz2024density}.

Note that the error terms in Theorem~\ref{thm: main} arising from $\gamma_T$ and $S_n$ are decoupled. In particular, by fixing either a closed geodesic or a totally geodesic hypersurface, we obtain effective equidistribution results in each case. These follow directly from the proof of Theorem~\ref{thm: main} and its generalization, Theorem~\ref{thm: general}.

\begin{Cor}
Let $M$ be an arithmetic hyperbolic manifold of dimension $d\geq 3$. Then there exist positive constants $\epsilon_2,q$, and a positive function $C(\iota)$, depending only on the geometry of $M$, such that the following holds. Let $\gamma \subset M$ be a closed geodesic, and let $\ell$ denote the length measure on $\gamma$, induced by the hyperbolic metric on $M$. Consider an infinite sequence $\{S_n\}$ of distinct totally geodesic hypersurfaces in $M$. Suppose further that $\gamma$ is disjoint from the cuspidal neighborhood $M_{\iota}$. Then for any $f \in C^{\infty}_c(M)$, we have

$$
\left|\int_{M} f \, d\ell^1-\int_{M} f \, d\,I^1(\gamma,S_n)\right| \leq C(\iota)\vol(S_n)^{-\epsilon_2}\Sob_q(f).
$$
\end{Cor}
We remark here that the dependence of the constant $C(\iota)$ on $\iota$ here comes from Theorem~\ref{thm: general}, where one has to restrict oneself to a compact subset.
\begin{Cor}
Let $M$ be a hyperbolic manifold of dimension $d \geq 2$. There exist positive constants $C,\epsilon_1,q$ depending only on the geometry of $M$ so that the following holds. Given geodesic hypersurface $S$ in $M$, let $\alpha$ denote the volume measure on $S$ (obtained from the hyperbolic metric on $M$). Then for any $f \in C^{\infty}_c(M)$, we have
$$
\left|\int_{M} f \, d\alpha^1-\int_{M} f \, d\,I^1(\gamma_T,S)\right| \leq Ce^{-\epsilon_1T}\Sob_{q}(f).
$$ 
\end{Cor}
In this corollary, $M$ need not be arithmetic; only the existence of one geodesic hypersurface is necessary. Examples of non-arithmetic hyperbolic $d$-manifolds containing a geodesic hypersurface can also be found in \cite{Non-arithmetic}. Note that for the case $d=2$, the hypersurface $S$ is simply a closed geodesic.

\begin{Rmk}
    We cannot replace $\gamma_T$ with a sequence of distinct closed geodesics, as they may not be equidistributed on $M$. For instance, if $M$ is a surface, consider the closed geodesics contained within a proper subsurface of $M$. On the other hand, when such a sequence is equidistributed in the compact case, our method does yield the non-effective version of the main result, but the error term cannot be made uniform; see Example~\ref{eg: geodesic} in  \S\ref{subsec: counterexamples}.

    However, in the presence of cusps, we may not even have a (non-effective) equidistribution result for an equidistributed sequence of closed geodesics; see Example \ref{ex: cusp} in \S \ref{subsec: counterexamples}.
\end{Rmk}

\paragraph{Generalization}
We next discuss similar results for properly immersed totally geodesic submanifolds (or simply \emph{geodesic submanifolds}) of complementary dimensions. To avoid the case of closed geodesics already discussed, assume $d\ge4$ and $2\le k\le d-2$.

A geodesic submanifold of dimension $k$ in $M$ is called \emph{maximal} if it is not contained in any geodesic submanifold of dimension larger than $k$ other than $M$ itself. Similar to geodesic hypersurfaces, any infinite sequence of distinct maximal geodesic $k$-submanifolds are equidistributed (see Proposition~3.1 in \cite{arithmetic_geodesic} and its proof). Moreover, if $M$ contains such a sequence, it is arithmetic \cite{arithmetic_geodesic}. We similarly have the converse by the density of the commensurator: if $M$ is arithmetic and contains at least one maximal geodesic $k$-submanifold, then it contains infinitely many.

Suppose that $M$ is an arithmetic hyperbolic $d$-manifold containing at least one maximal geodesic $k$-submanifold, and one maximal geodesic $(d-k)$-submanifold. Let $\{\Gamma_m\}, \{S_n\}$ be sequences of distinct maximal geodesic $k$- and $(d-k)$-submanifolds respectively. Let $I_{\ge\iota}(\Gamma_m,S_n)$ denote the sum of delta measures supported on the transverse intersection points between $\Gamma_m$ and $S_n$ contained in $M_{\ge\iota}:=\overline{M\backslash M_\iota}$, each weighted by its multiplicity. Denote also by $\vol_{\ge\iota}$ the volume measure restrictued to $M_{\ge\iota}$.
\begin{Thm}\label{thm: general}
    Let $M$ be an arithmetic hyperbolic manifold of dimension $d\ge 4$. Suppose $2\le k\le d-2$. Let $\{\Gamma_m\}, \{S_n\}$ be sequences of distinct maximal geodesic $k$- and $(d-k)$-submanifolds respectively. Then there exist positive constants $\epsilon_2,\epsilon_2',q$ and a positive function $C(\iota)$ depending only on the geometry of $M$, so that for any $f\in C_c^\infty(M)$, we have
    $$
    \left|\int_{M} f \, d\vol^1_{\ge\iota}-\int_{M} f \, d\,I^1_{\ge\iota}(\Gamma_m,S_n)\right| \leq C(\iota)(\vol(\Gamma_m)^{-\epsilon_2}+\vol(S_n)^{-\epsilon_2'})\Sob_{q}(f).
    $$
    Moreover, if we fix a $k$-submanifold $\Gamma$ with volume form $\vol_{\Gamma}$, we have
    $$\left|\int_{M} f \, d\vol^1_{\Gamma\cap M_{\ge\iota}}-\int_{M} f \, d\,I^1_{\ge\iota}(\Gamma,S_n)\right| \leq C(\iota)\vol(S_n)^{-\epsilon_2'} \Sob_{q}(f).$$
\end{Thm}
\begin{Rmk}
    Unlike Theorem~\ref{thm: main}, this result is essentially equidistribution restricted to the compact part $M_{\ge\iota}$. When $M$ is noncompact, this does not necessarily imply equidistribution over the entire $M$, since there could be escape of mass towards the cusp, i.e. a definite portion of intersection points in shrinking cuspidal neighborhoods. Our current method does not seem to yield a good control of the portion of intersection points in a cuspidal neighborhood unless closed geodesics are involved; see the discussion at the end of \S\ref{sec:cusp_inter}.
\end{Rmk}
\begin{Rmk}
    Our method applies in general without the maximal assumption. In this case, the sequence of geodesic $k$-submanifolds $\{\Gamma_m\}$ may equidistribute to a geodesic submanifold $N$ of intermediate dimension. Equidistribution of intersection points (to the submanifold $N$) still holds. However we can no longer guarantee that the error term depends only on the geometry of $M$; see Example~\ref{eg: no_maximal} and Example~\ref{eg: intermediate} in \S\ref{subsec: counterexamples}.
\end{Rmk}

\paragraph{Counterexamples}
In \S \ref{subsec: counterexamples}, we describe several counterexamples to illustrate the necessity of certain assumptions in our main results, as mentioned before. Notably, after removing some of the assumptions our proof may still yield non-effective versions of the results, but we cannot obtain any error term of the form $O(e_1(\vol(\Gamma_m))+e_2(\vol(S_n)))\Sob_q(f)$ with functions $e_1$ and $e_2$ satisfying $e_i(x)\to0$ as $x\to\infty$, so that the implied constant in $O(\cdot)$ and the functions $e_1$ and $e_2$ depend only on the geometry of $M$.

While we focus on unweighted sequences, our method in fact gives equidistribution of intersection points in the compact case for any sequence of weighted sums of closed geodesic submanifolds that equidistribute to the Liouville measure (however, the error term cannot generally be chosen to be uniform). In contrast, as previously mentioned, when $M$ has finite volume but is not compact, we can not in general conclude an equidistribution result for intersection points, unless we work with specific sequences. In the case of $k=1$, the sequence $\{\gamma_T\}$ works; but in \S \ref{subsec: counterexamples}, we include an example showing that not every equidistributed sequence does. We remark that this example is closely related to the discontinuity of the intersection form on the space of geodesic currents on a non-compact hyperbolic surface of finite area; cf.\ \cite{torkaman2024intersection}.

\paragraph{Joint equidistribution}
The main results can be generalized to a joint equidistribution result taking into account how the geodesic submanifolds intersect. For simplicity, we only state the result for the compact case.

Let $\Conf(k,d)$ be the configuration space of pairs of $k$-planes and $(d-k)$-planes in $\mathbb{R}^d$ up to rigid rotation (see \S\ref{sec:joint}). Each intersection point $P$ of a geodesic $k$-submanifold $\Gamma$ and a geodesic $(d-k)$-submanifold $S$ in fact determines a point in $(P,(T_P\Gamma,T_PS))\in M\times \Conf(k,d)$, by considering the tangent planes to $\Gamma$ and $S$ at $P$.

Let $I_{\text{conf}}(\Gamma,S)$ be the sum of delta measures at $(P,(T_P\Gamma,T_PS))$ over each transverse intersection point $P$ between $\Gamma$ and $S$. We have the following joint equidistribution theorem (for simplicity, we only state the case $2\le k\le d-2$).
\begin{Thm}\label{thm: joint_equidistr}
    There exists a probability measure $\mu_{k,d}$ on $\Conf(k,d)$ so that the following holds. Let $M$ be a closed arithmetic hyperbolic manifold of dimension $d$, and let $\{\Gamma_m\}$ and $\{S_n\}$ be sequences of distinct geodesic $k$- and $(d-k)$-submanifolds respectively. There exist positive constants $C,\epsilon_2,\epsilon_2',q$ depending only on the geometry of $M$, so that for any $f\in C^\infty(M)$ and $h\in C^\infty(\Conf(k,d))$, we have
    \begin{align*}
        &\left|\int_{M\times\Conf(k,d)} f(x)h(\theta)\,dI^1_{\text{conf}}(\Gamma_m,S_n)-\int_Mf(x)\,d\vol^1(x)\int_{\Conf(k,d)}h(\theta)\,d\mu_{k,d}(\theta)\right|\\
        &\qquad\qquad\qquad\qquad\le C(\vol(\Gamma_m)^{-\epsilon_2}+\vol(S_n)^{-\epsilon_2'})\Sob_q(f)\Sob_q(h).
    \end{align*}
    Moreover, if we fix a geodesic $(d-k)$-submanifold $S$, then
    \begin{align*}
        &\left|\int_{M\times\Conf(k,d)} f(x)h(\theta)\,dI^1_{\text{conf}}(\Gamma_m,S)-\int_Mf(x)\,d\vol^1_S(x)\int_{\Conf(k,d)}h(\theta)\,d\mu_{k,d}(\theta)\right|\\
        &\qquad\qquad\qquad\qquad\le C\vol(\Gamma_m)^{-\epsilon_2}\Sob_q(f)\Sob_q(h).
    \end{align*}
\end{Thm}
Setting $h\equiv1$, we obtain Theorem~\ref{thm: general} in the compact case. Note that setting $f\equiv 1$, we also obtain an equidistribution of intersecting configuration in $\Conf(k,d)$ with respect to $\mu_{k,d}$.

\paragraph{Discussion on related works}
\emph{Topological} equidistribution (i.e.~intersection points are dense in $M$), instead of \emph{measure-theoretic} equidistribution, is much easier to establish, and follows from topological versions of Theorem~\ref{thm: mixing.geod.3} and Theorem~\ref{thm: mixing.surface} (see \cite{ratner, shah}).

Filip, Fisher and Lowe recently showed that if a closed, real-analytic Riemannian manifold with negative sectional curvature of dimension $d\ge3$ contains infinitely many closed totally geodesic immersed hypersurfaces, then it must be hyperbolic, and thus arithmetic \cite{filip2024finiteness}. Therefore, Theorem~\ref{thm: main} covers all cases of interest even if we move to this more general setting. It is conjectured in \cite{filip2024finiteness} that a similar result holds for maximal totally geodesic submanifolds of dimension $\ge2$ as well. If so, it would mean Theorem~\ref{thm: general} also covers all cases of interest in this general setting.

In dimension 2, Lalley proved that the self-intersection points of ``most'' closed geodesics on compact negatively curved surfaces approximately equidistribute, without explicit error terms \cite{lalley_equiv}, as part of the pioneering work on statistical properties of these intersection points \cite{lalley2009self, chas_lalley, lalley_stat_reg}. Recently, Katz established an effective equidistribution result for intersection points of closed geodesics in the same setting \cite{katz2024density}.

The first-named author also proved equidistribution of intersection points between closed geodesics on hyperbolic surfaces of \emph{finite area}, without explicit error terms \cite{torkaman2024intersection}. The method there relies on the theory of geodesic currents (cf.~\cite{Bon.gc, Bon.Tch}) to deal with the compact part, plus control of intersection points in cuspidal neighborhoods. We expect that a similar approach based on geodesic currents would also lead to equidistribution for the compact case in higher dimensions, although without error terms. One of the key contributions of \cite{torkaman2024intersection} is to control the escape of mass to infinity through cuspidal neighborhoods, which we adapt here to obtain Theorem~\ref{thm: main} for the non-compact case.

On the other hand, Jung and Sardari proved a joint distribution result with explicit exponents in the error term for the modular surface \cite{MR4595382}, which is noncompact. Our approach follows some of the ideas there to define the intersection kernel (see Equation~\ref{def: K}) in the noncompact setting. They were able to obtain explicit (albeit non-optimal) exponents using tools from analytic number theory.

We end the discussion by pointing out some settings where one might expect a similar equidistribution result, as possible directions of future research.
\begin{itemize}[itemsep=0mm, topsep=0mm]
    \item The case of geodesic $k$ and $(d-k)$ submanifolds in non-compact arithmetic hyperbolic manifolds, where $2\le k\le d-2$, without restricting to compact subsets. As mentioned before, the key difficulty here is to control the escape of mass to the cusps.
    
    \item Intersection points between closed geodesics and minimal hypersurfaces for a negatively curved Riemannian manifold. Recently, it was proved that for a generic Riemannian metric on a smooth manifold of dimension between 3 and 7, there exists a sequence of minimal hypersurfaces that is equidistributed \cite{minimal_equidistr}. If effectivized, it could lead to a version of the main theorem in this setting.
    
    \item Intersection points between geodesic submanifolds of complementary \emph{complex} dimension in a complex-hyperbolic manifold. The equidistribution results cited in \S\ref{sec:equidistri} also hold in this setting (at least the ones without effective error terms); the analysis in \S\ref{sec:estimates1} and \S\ref{sec:estimates2} however needs to be adapted.
    
    \item Intersection points between closed geodesics on a geometrically finite hyperbolic surface of infinite area. In this case, closed geodesics equidistribute with respect to the Bowen-Margulis-Sullivan measure on the unit tangent bundle \cite{roblin}, which is supported on a fractal set. Points on the surface seem less ``homogeneous'' compared to the case of finite area (e.g.\ no closed geodesics intersect the components of the convex core boundary, which is a finite collection of simple closed geodesics). Still, all other closed geodesics have closed geodesics intersecting them; thus, some modifications are needed.
\end{itemize}

\paragraph{Big $O$, small $o$}
Throughout the paper, by $f=O(g)$ we mean that there exists a constant $C$ depending only on the geometry of $M$ so that $f\le C g$. By $f=o(g)$, we mean that there exists a function $h$ depending only on the geometry of $M$ so that $f/g\le h$ and $h\to0$, as the input goes to certain limit (usually $0$ or $\infty$). Moreover, we write $f \asymp g$ if $f=O(g)$ and $g=O(f)$.

\paragraph{Acknowledgment}
We thank Alex Eskin for valuable suggestions and enlightening discussions. We also thank Junehyuk Jung for pointing out the method in \cite{MR4595382} can be adapted to deal with the noncompact case here. Y.~Z. is partialy supported by an AMS-Simons Travel Grant.

\section{The intersection kernel}\label{sec: inter_kernel}
Throughout this section, let $M$ be a $d$-dimensional oriented complete hyperbolic manifold of finite-volume containing a geodesic $k$-submanifold and a geodesic $(d-k)$-submanifold, presented as $M\cong\Pi\backslash\mathbb{H}^d$, where $\Pi\subseteq\isom^+(\mathbb{H}^d)$ is a discrete torsion-free subgroup of orientation-preserving isometries. We remark that the oriented assumption is simply for convenience; all results remain valid by considering the oriented double cover for the nonorientable case.

\paragraph{Frames, submanifolds, and measures}
Let $\mathcal{F}(M)$ be the frame bundle over $M$, consisting of orthonormal tangent frames $(u_1,\dots,u_d)$ at each point $x\in M$, consistent with the orientation of $M$.
Consider the projection $\pi:\mathcal{F}(M) \to M$ that sends a frame to its base point. Note that every continuous function $f\in C(M)$ can be pulled back to a continuous function $f\circ\pi\in C(\mathcal{F}(M))$. For simplicity, we use the same notation to denote the two functions; it will be clear from context which one we refer to.

Given $1\le k\le d-1$, for any immersed totally geodesic hyperbolic $k$-plane $\mathcal{P}$ in $M$ (which is not necessarily properly immersed), we say a frame $(u_1,\dots,u_d)$ is tangent to $\mathcal{P}$ if the base point lies on $\mathcal{P}$ and the first $k$ vectors in the frame are tangent to $\mathcal{P}$. We denote the $k$-dimensional hyperbolic ball of radius $\delta$ contained in $\mathcal{P}$ centered at $x=\pi(v)$ by $B^k_\delta(v)$, or $B^{\mathcal{P}}_\delta(x)$ if we only want to specify the base point. Note that $B_\delta^d(v):=B_\delta(\pi(v))$ is simply the open ball of radius $\delta$ centered at $\pi(v)$ in $M$.

Denote by $\phi_\tau$ the geodesic flow on the unit tangent bundle of $M$. We also have an induced flow on the frame bundle along the direction of the first vector by parallel transport, which we also denote by $\phi_\tau$. Note that $\pi( \{ \phi_\tau(v): \tau\in (-\delta,\delta)\})=B^1_\delta(v)$ is a geodesic segment of length $2\delta$ with midpoint $\pi(v)$. 

Let $\mathbf{S}^{d-1}\subset\mathbb{R}^{d}$ be the unit sphere in dimension $d$, with the canonical spherical volume measure $\nu_{d-1}$. Set $\alpha_{d-1}=|\nu_{d-1}|$. The set of orthonormal frames $(u_1,\ldots,u_d)$ in $\mathbb{R}^d$ supports a natural measure, locally given by $d\nu_{d-1}\times\cdots\times d\nu_1$.

Given a geodesic $k$-submanifold $N$ in $M$, a corresponding Borel measure $\mu_N$ on $\mathcal{F}(M)$ can be defined as follows.
Consider $\mathcal{F}_{M}(N) \subseteq \mathcal{F}(M)$ containing frames $(u_1,\dots,u_d)$ which the first $k$ vectors are tangent to $N$.
Note that $\mathcal{F}_{M}(N)$ has one or two connected components, depending on whether $N$ is orientable or not.
Define a measure $\widetilde{\mu_N}$ supported on $\mathcal{F}_{M}(N)$ locally by $\frac12d\vol_N\times (d\nu_{k-1}/\alpha_{k-1})\times\cdots\times(d\nu_1/\alpha_1)\times (d\nu_{d-k-1}/\alpha_{d-k-1})\times\cdots\times (d\nu_1/\alpha_1)$, where $d\vol_N$ is the hyperbolic volume measure on $N$, followed by $k-1$ factors corresponding to the $k$ vectors tangent to $N$ (given $k-1$ vectors, the $k$th vector is determined uniquely), and then by $d-k-1$ factors corresponding to the remaining vectors. Note that the factor $\frac12$ is introduced because of the two choices of local orientation.
Now define $\mu_{N}(A) :=\widetilde{\mu_{N}}(A\cap \mathcal{F}_{M}(N))$ for any Borel set $A\subseteq\mathcal{F}(M)$.

In particular, this construction works for a closed geodesic $\gamma$ (when $k=1$). We can extend the definition to any weighted sum of closed geodesics. When $k=d$, the measure we obtain is (up to a constant) the \emph{Liouville measure} on $\mathcal{F}(M)$, which we denote by $L_M$. Similar construction gives a uniform measure on the unit tangent bundle $T_1(M)$, which we also call the Liouville measure and denote by $L_M$. It will be clear from context which one we refer to.

\paragraph{Intersection measure}
Let $\Gamma$ be a geodesic $k$-submanifold and $S$ a geodesic $(d-k)$-sub\-manifold in $M$. When $k=1$ (or $d-k=1$), we also allow $\Gamma$ (or $S$) to be a (formal) sum of closed geodesics. Let $\mathscr{P}(\Gamma,S)$ denote the set of transverse intersection points between (components of) $\Gamma$ and $S$.

We define the \emph{intersection measure} $I(\Gamma,S)$ between $\Gamma$ and $S$ on $M$ as the sum of the delta measures at the transverse intersection points. It is considered with multiplicities; the weight we assign to the point is the minimum number of simple intersection points when we homotopically deform $\Gamma$ in a small neighborhood. For example, when a geodesic arc passes through the self-intersection point of a figure eight, the multiplicity at this intersection point between the figure eight and the arc is two.

\paragraph{Intersection kernel}
Unless otherwise stated, when we write $\vol(X)$ for some $k$-dimensional space (in $M$ or its universal cover $\mathbb{H}^d$), we mean the $k$-dimensional volume on $X$ induced from the hyperbolic metric. We also use $X\pitchfork Y$ to denote that they intersect transversely. 

We first define an intersection kernel in the universal cover $\mathbb{H}^d$. Given $u,v\in\mathcal{F}(\mathbb{H}^d)$, define
$$
\mathbb{K}_{\delta}(u,v)=
\begin{cases}
    \frac{1}{c_\delta}&\text{if $\overline{B_\delta^k(u)}\pitchfork \overline{B^{d-k}_{\delta}(v)}$,} \\
    0 &\text{otherwise.}
\end{cases}
$$
where $c_{\delta}=\vol(B_\delta^k(u))\times \vol(B^{d-k}_{\delta}(v))\asymp \delta^{d}$. Note that $\vol(B_\delta^k(u))$ is independent of $u$, which we simply denote by $\vol_k(\delta)$. Although the kernel $\mathbb{K}_{\delta}$ depends on $k$ we dropped index $k$ as it will be clear from the context. Clearly, $\mathbb{K}_\delta$ is $\isom(\mathbb{H}^d)$-invariant, i.e.\ $\mathbb{K}_\delta(g.u,g.v)=\mathbb{K}_\delta(u,v)$ for any $g\in\isom(\mathbb{H}^d)$.

We are now ready to define the \emph{intersection kernel} $K_{\delta}:\mathcal{F}(M)\times \mathcal{F}(M) \to \mathbb{R}$ on $M\cong\Pi\backslash\mathbb{H}^d$ as follows:
\begin{equation}\label{def: K}
K_{\delta}(u,v)=\sum_{g\in\Pi}\mathbb{K}_\delta(g.\tilde u,\tilde v)
\end{equation}
where $\tilde u,\tilde v\in\mathcal{F}(\mathbb{H}^d)$ are any lifts of $u,v\in\mathcal{F}(M)$. By invariance, the expression does not depend on the choice of the lifts. By proper discontinuity of the action of $\Pi$ on $\mathbb{H}^d$, for any fixed $\tilde u,\tilde v\in\mathcal{F}(\mathbb{H}^d)$, there are only finitely many $g\in\Pi$ so that $\mathbb{K}_\delta(g.\tilde u,\tilde v)$ is nonzero. Hence, the kernel $K_\delta$ is well-defined.

Note that when $M$ is compact, by choosing $\delta$ small enough, we may guarantee $\mathbb{K}_\delta(g.\tilde u,\tilde v)\neq0$ for at most one $g\in\Pi$. In particular, it is possible to define the kernel directly on $M$ for $\delta$ small enough without going to the universal cover. This is not possible in the presence of cusps.

This definition of the intersection kernel is motivated by the following proposition, which helps us to separate the contributions of $k$- and $(d-k)$-dimensional submanifolds to the intersection number.
\begin{Prop}\label{prop: intersection.kernel} 
Let $\Gamma$ be a geodesic $k$-submanifold and $S$ a geodesic $(d-k)$-submanifold in $M$ (not necessarily compact).
For any $f \in C_c(M)$, we have
\begin{equation}\label{equ: intersection.kernel}
\int_{M} f(x) \, dI(\Gamma,S)=\lim_{\delta \to 0} \int_{\mathcal{F}(M)} \int_{\mathcal{F}(M)} f(v)K_{\delta}(u,v) \, d\mu_{\Gamma}(u)d\mu_{S}(v).
\end{equation} 
\end{Prop}

\begin{proof}
    The left-hand side of Equation~\ref{equ: intersection.kernel} is precisely the sum of $f$ values at the intersection points, considered with the multiplicities. Let $P$ be an intersection point that, for simplicity, we can assume is simple (i.e.\ has multiplicity one). We want to show the contribution of the point $P$ to the right-hand side is $f(P)$ too.
    
    Since $f$ is compactly supported, we can fix $\delta>0$ small enough so that the closed $\delta$-neighborhoods around the intersection points in the support of $f$ are disjoint from each other.
    Since $f$ is furthermore continuous, there exists $\epsilon>0$ so that $|f(x)-f(y)|\le\epsilon$ whenever $|x-y|\le\delta$. Moreover, we can choose $\epsilon\to 0$ as $\delta\to 0$.

    Our assumption on $\delta$ guarantees that for any frame $v$ tangent to $S$, the set $\overline{B_\delta^{d-k}(v)}$ contains at most one intersection points. Indeed, if both $P,Q\in \overline{B_\delta^{d-k}(v)}$, then $d(P,Q)\le2\delta$ and hence $\overline{B_\delta(P)}\cap \overline{B_\delta(Q)}\neq\emptyset$, a contradiction.

    For any $\pi(v)\in \overline{B^S_\delta(P)}$ with first $d-k$ vectors of $v$ tangent to $S$, we have $\displaystyle \int_{\mathcal{F}(M)}K_\delta(u,v)d\mu_\Gamma(u)\allowbreak=\vol_k(\delta)/c_\delta$ and $|f(v)-f(P)|\leq \epsilon$. Thus
    \begin{equation*}
 \left|\int_{\pi^{-1}(B_\delta^S(P))}\int_{\mathcal{F}(M)}f(v)K_\delta(u,v)d\mu_\Gamma(u)d\mu_S(v)-f(P)\right|\le \frac{\vol_k(\delta)\vol(B_\delta^S(P))\epsilon}{c_\delta}=\epsilon.
    \end{equation*}
    Therefore, the contribution of $P$ to the right-hand side of Equation \ref{equ: intersection.kernel} is in $\lim_{\delta \to 0} [f(P)-\epsilon,f(P)+\epsilon]$ so it is $f(P)$, as desired.
\end{proof}

\paragraph{Regularity and growth}
More generally, suppose $\mathbb{K}:\mathcal{F}(\mathbb{H}^d)\times\mathcal{F}(\mathbb{H}^d)\to\mathbb{R}$ is a bounded function satisfying
\begin{itemize}
    \item $\mathbb{K}(g.u,g.v)=\mathbb{K}(u,v)$ for any $g\in\isom(\mathbb{H}^d)$, and
    \item given a fixed $v$, the function $u\mapsto\mathbb{K}(u,v)$ has compact support. We denote the diameter of this support by $\diam(\mathbb{K})$.
\end{itemize}
We can then define the associated function $K:\mathcal{F}(M)\times\mathcal{F}(M)\to\mathbb{R}$ by
$$K(u,v)=\sum_{g\in\Pi}\mathbb{K}(g.\tilde u,\tilde v),$$
where $\tilde u,\tilde v$ are any lifts of $u,v$. We have the following properties of the function $K$.
\begin{Prop}\label{prop: growth}
    \begin{enumerate}[label=\normalfont{(\arabic*)}]
        \item If $\mathbb{K}$ is of class $C^l$, then so is $K$.
        \item There exists a constant $C$ depending only on the geometry of $M$ so that
        $$\sup_{\pi(v)\in M_{\ge\iota}}|K(u,v)|\le C\lceil \diam(\mathbb{K})/\iota\rceil^{d-1}\sup|\mathbb{K}|.$$
        Similar statements hold for all derivatives of $K$ as well.
        \item We have
        $$\int_{\mathcal{F}(M)}K(u,v)dL_M(u)=\int_{\mathcal{F}(\mathbb{H}^d)}\mathbb{K}(\tilde u,\tilde v)dL_{\mathbb{H}^d}(\tilde u)$$
        where $\tilde v$ is any lift of $v$. In particular, it is a constant independent of $v$.
    \end{enumerate}
\end{Prop}
\begin{proof}
    For part (1), given any $u,v\in\mathcal{F}(M)$, choose small neighborhoods $U$ and $V$ of them. Take any lifts $\tilde U,\tilde V\subseteq\mathcal{F}(\mathbb{H}^d)$. By proper discontinuity, there exist only finitely many $g\in\Pi$ so that $\mathbb{K}$ is nonzero on $g.U\times V$. In particular, $K$ is the sum of finitely many functions of class $C^l$ in $U\times V$, and is hence of class $C^l$ there as well.

    For Part (2), for fixed $\tilde u, \tilde v$, we only need to control the number of $g$ so that $\mathbb{K}(g.\tilde u,\tilde v)$ is nonzero. When $v$ is a frame over a point $\pi(v)$ deep in the cusp, then it is easy to see that $g$ must be parabolic elements fixing the corresponding lift of the cusp. The estimate then follows.

    For Part (3), choose a fundamental domain $F$ of $\Pi$ in $\mathcal{F}(\mathbb{H}^d)$. Then
    \begin{align*}
        \int_{\mathcal{F}(M)}K(u,v)dL_M(u)&=\int_FK(\pi(\tilde u),v)dL_{\mathbb{H}^d}(\tilde u)=\int_{F}\sum_{g\in\Pi}\mathbb{K}(g.\tilde u,\tilde v)dL_{\mathbb{H}^d}(\tilde u)\\
        &=\sum_{g\in\Pi}\int_{g\cdot F}\mathbb{K}(\tilde u,\tilde v)dL_{\mathbb{H}^d}(\tilde u)=\int_{\mathcal{F}(\mathbb{H}^d)}\mathbb{K}(\tilde u,\tilde v)dL_{\mathbb{H}^d}(\tilde u)
    \end{align*}
    which is a constant by invariance of $\mathbb{K}$ and $L_{\mathbb{H}^d}$ under isometries.
\end{proof}

\section{Effective equidistribution of geodesic submanifolds}\label{sec:equidistri}

In this section, we recall some effective equidistribution theorems needed in the proof of Theorem \ref{thm: main} and Theorem~\ref{thm: general}. We also discuss several counterexamples to illustrate the necessity of certain assumptions in our main results. 
\subsection{Equidistribution results}
The first result gives an effective equidistribution of closed geodesics. Recall that $\gamma_T$ is the sum of all closed geodesics in $M$ of length $\le T$.
\begin{Thm}[Theorem~5.1 in \cite{effective.closedgeo}]\label{thm: mixing.geod.3}
    Let $M$ be a complete hyperbolic manifold of finite volume. There exist constants $\epsilon_1, q>0$ such that for any bounded $f \in C^\infty(T_1(M))$ we have
    $$
    \int_{T_1(M)} f \, d\mu_{\gamma_T}^1=\int_{T_1(M)} f \, dL_M^1+O(\Sob_q(f)e^{-\epsilon_1 T}). 
    $$
    where $\Sob_q$ is a degree $q$ Sobolev norm of $f$.
\end{Thm}
While this theorem is stated here for the unit tangent bundle instead of the frame bundle, we only need this result when $f$ is constant on the fibers of $\mathcal{F}(M)\to T_1(M)$.

\begin{Thm}\label{thm: mixing.surface}
Let $M$ be an arithmetic hyperbolic manifold, and $k\ge 2$. There exist constants $\epsilon_2, q>0$ such that for any maximal geodesic $k$-submanifold $S$ and any function $f \in C_c^{\infty}(\mathcal{F}(M))$ we have: 
$$
\int_{\mathcal{F}(M)} f \, d\mu_S^1=\int_{\mathcal{F}(M)} f \, dL_M^1+ O(\Sob_q(f)\vol(S)^{-\epsilon_2}),
$$
where $\Sob_q$ is a degree $q$ Sobolev norm of $f$.
\end{Thm}
This theorem follows from the results in \cite{effective.surface, effective.surface2}, most notably Theorem~1.3 in \cite{effective.surface} and Theorem~1.7 in \cite{effective.surface2}. We remark that the Sobolev norm in these results is a weighted $L^2$ norm, with a weight tending to infinity in the cusp (in fact, the weight is $\asymp\inj(x)^{-q}$), where $\inj(x)$ denotes the injectivity radius at $x$; see the relevant discussions in \cite{effective.surface} for more details. On the other hand, in the compact case, the Sobolev norms of the same degree are equivalent (up to a multiplicative constant), and it does not matter which one we choose.

The remainder of the section is devoted to briefly explain the discrepancies between the theorems stated above and their original versions in the cited papers.

\paragraph{Homogeneous dynamics: geodesic flow}
Let $G_d=\SO(d,1)$, and $G_d^+=\SO^+(d,1)$ be the identity component of $G_d$. Set $K_d\cong \SO(d)$ to be a maximal compact subgroup of $G_d^+$. The homogeneous space $G^+_d/K_d$ is endowed with a left $G^+_d$-invariant metric, normalized to have constant sectional curvature $-1$. This space is isometrically identified with $\mathbb{H}^d$, where $G^+_d$ acts as orientation preserving isometries.

Fix $o\in\mathbb{H}^d$ and a frame $v_0$ at $o$. Note that $G^+_d$ acts transitively on the frame bundle $\mathcal{F}(\mathbb{H}^d)$. The stabilizer of the frame $v_0$ is trivial. Thus the $G^+_d$ is identified with $\mathcal{F}(\mathbb{H}^d)$.

Let $T_1(\mathbb{H}^d)$ be the unit tangent bundle over $\mathbb{H}^d$. The stabilizer of a fixed tangent vector at $o$ is a compact subgroup $M_d$ of $K_d$, and so $T_1(\mathbb{H}^d)$ is identified with $G^+_d/M_d$. There exists a one-parameter subgroup $A=\{a_t:t\in\mathbb{R}\}$ of $G_d$ consisting of diagonalizable elements which commute with $M_d$. The right action of $a_t$ on $G^+_d/M_d$ corresponds to the geodesic flow $\phi_t$ on $T_1(\mathbb{H}^d)$.

Any hyperbolic $d$-manifold can be presented as $\Pi\backslash\mathbb{H}^d$ where $\Pi\subseteq G^+_d$ is a discrete subgroup without torsion. Our assumption on $M$ means that $\Pi$ is an arithmetic lattice. The quotient $\Pi\backslash G^+_d$ is naturally identified with the frame bundle $\mathcal{F}(M)$, and $\Pi\backslash G^+_d/M_d$ is identified with the unit tangent bundle $T_1(M)$. Moreover, the right action of $a_t$ on $\Pi\backslash G^+_d/M_d$ corresponds to the geodesic flow $\phi_t$ on $T_1(M)$.

\begin{proof}[Sketch of Theorem~\ref{thm: mixing.geod.3}.]
    By \cite[Theorem~5.1]{effective.closedgeo}, given any bounded smooth function $f$ on $T_1(M)$, we have $d\cdot e^{-dT}\mu_{\gamma_T}(f)=L^1_M(f)+O(\Sob_q(f)e^{-\epsilon_1T})$. Setting $f=1$, we conclude that $d\cdot e^{-dT}|\mu_{\gamma_T}|=1+O(\Sob_q(f)e^{-\epsilon_1T})$. Combining the two, we have Theorem~\ref{thm: mixing.geod.3}.
\end{proof}

\paragraph{Homogeneous dynamics: geodesic submanifolds}
Now for any $2\le k\le d-1$, $G_k$ is contained in $G^+_d$ as a subgroup. Orbits of the right action of $G_k$ on $G^+_d\cong\mathcal{F}(H^d)$ and $\Pi\backslash G^+_d\cong\mathcal{F}(M)$ are precisely frames over properly immersed totally geodesic $k$-planes in $\mathbb{H}^d$ and $M$, where two different connected components of $G_k$ give two different local orientations. In particular, geodesic $k$-submanifolds correspond to closed $G_k$-orbits.

\begin{proof}[Sketch of Theorem~\ref{thm: mixing.surface}]
    We first prove the statement with the extra assumption that $\Pi$ is a congruence subgroup. When $k=d-1$, the result follows directly from Theorem 1.3 in \cite{effective.surface}, as in this case, $G^+_{d-1}$ has finite centralizer in $G^+_d$. (Geometrically speaking, any tangent frame to a geodesic hypersurface uniquely determines a frame in $M$ with compatible orientation, so closed $G^+_{d-1}$-orbits are rigid.)

    More generally, we can apply Theorem~1.7 in \cite{effective.surface2}. Indeed, note that by our assumption, the closed $G^+_k$-orbit $xG^+_k$ corresponding to the maximal geodesic $k$-submanifold $S$ is not, up to compact factor, contained in the orbit of any intermediate closed subgroup of $G^+_d$ (see \cite[Lemma~3.2]{arithmetic_geodesic}). Hence we conclude that the error term is of the form $O(\Sob_q(f)\disc(xG^+_k)^{\epsilon_2})$ for some $q,\epsilon_2$ depending only on the geometry of $M$. Here $\disc(xG^+_k)$ denotes the discriminant of the closed orbit, and by \cite[Proposition 17.1]{effective.surface} it is comparable to $\vol(S)$ up to a positive exponent. Therefore by changing $\epsilon_2$ if necessary, we have the desired result in this case.
    
    It remains to argue that in our setting, we may drop the congruence assumption. This can be done following the argument of \cite[\S1.6.1]{effective.surface}. First note that we may assume $\Pi$ is contained in a congruence group $\Lambda$ by taking a finite index subgroup. Clearly any maximal geodesic $k$-submanifold in $M$ projects to a maximal geodesic $k$-submanifold in $\Lambda\backslash\mathbb{H}^d$. We can then apply the theorem in the congruence case and argue exactly as \cite[\S1.6.1]{effective.surface} to obtain the desired error term.
\end{proof}

\subsection{Counterexamples}\label{subsec: counterexamples}

In this subsection, we provide several counterexamples to demonstrate why certain assumptions in our main theorems are essential. In particular, if some of these assumptions are dropped, the arguments may still lead to non-effective versions of the results. However, we can no longer obtain an error bound of the form
\[
O\big(e_1(\vol(\Gamma_m)) + e_2(\vol(S_n))\big) \, \Sob_q(f),
\]
where the functions $e_1$ and $e_2$ tend to zero as $x \to \infty$, and both these functions, along with the implied constant in the $O(\cdot)$ notation, depend only on the geometry of $M$.

To simplify the discussion, we assume throughout this subsection that $M$ is closed, i.e.\ compact without boundary. 

\begin{Ex*}\label{eg: geodesic}

In general, we cannot replace \( \gamma_T \) with a sequence of distinct closed geodesics in Theorem~\ref{thm: main}, even if the sequence equidistributes in \( M \). To see this, suppose \( M \) contains an infinite sequence of distinct totally geodesic hypersurfaces \( \{S_n\} \), and let \( \{\gamma_m\} \) be a sequence of distinct closed geodesics entirely contained in \( S_1 \). Let \( \{\gamma_m'\} \) be a sequence of closed geodesics equidistributed in \( M \). For any \( K \in \mathbb{N} \), define a hybrid sequence \( \{\gamma_m^K\} \) by setting
\[
\gamma_m^K = 
\begin{cases}
\gamma_m & \text{if } m < K, \\
\gamma_m' & \text{if } m \geq K.
\end{cases}
\]
Then $I^1(\gamma_m^K, S_n) \to \vol_M^1$ as $m \to \infty$ (e.g., this follows from the proof without effective error terms). However, if  $K \gg 1$, for some large $m$ we have $\gamma_m^K \subset S_1$, so $I^1(\gamma_m^K, S_n)$ is far from $\vol_M^1$. However, the length $\ell(\gamma_m^K)$ may already be large, making the error term artificially small if the result was assumed to hold for arbitrary equidistributing sequences.
\end{Ex*}

\begin{Ex*}\label{eg: no_maximal}
    We cannot drop the maximal assumption in Theorem~\ref{thm: general} by a very similar construction. If the two sequences of submanifolds are not equidistributed then clearly we cannot expect the intersection points to be either. We provide a counterexample where the sequences remain equidistributed in $M$. Suppose $M$ contains a sequence of distinct maximal geodesic surfaces $\{S_n\}$. Let $N$ be a geodesic hypersurface in $M$. Suppose further that $N$ contains a sequence of distinct geodesic $(d-2)$-submanifolds $\{\Gamma_m\}$. By arithmeticity and density of commensurators, $M$ contains a sequence $\{\Gamma_m'\}$ of distinct geodesic $(d-2)$-submanifolds which equidistribute in $M$. Now as in the previous example, for any $K\in\mathbb{N}$, consider the sequence $\{\Gamma_m^K\}$ defined by $\Gamma_m^K=\Gamma_m$ when $m<K$ and $=\Gamma_m'$ when $m\ge K$. Then $I^1(\Gamma_m^K,S_n)\to\vol_M^1$ (e.g.~one may go through the proof without effective error terms). However, if $K\gg1$, then for some large $m$, we still have $\Gamma_m^K$ is contained in $N$, so $I^1(\Gamma_m^K,S_n)$ is far from being $\vol_M^1$, even though $\vol(\Gamma_m^K)$ is already very large.
\end{Ex*}
    
\begin{Ex*}\label{eg: intermediate}
    In Theorem~\ref{thm: general}, if we drop the maximal condition, then $\Gamma_m$ (up to a subsequence) equidistributes to a geodesic submanifold $N$ of $M$. One may ask whether we can have a uniform effective error with respect to the volume measure on $N$. We show that we can not. Indeed, suppose $M$ contains a sequence of distinct maximal geodesic surfaces $\{S_n\}$. Let $N_1,N_2$ be two distinct geodesic hypersurface in $M$. Suppose further that $N_i$ contains a sequence of distinct geodesic $(d-2)$-submanifolds $\{\Gamma_m^{(i)}\}$. As before, for any $K\in\mathbb{N}$, consider the sequence $\{\Gamma_m^K\}$ defined by $\Gamma_m^K=\Gamma^{(1)}_m$ when $m<K$ and $=\Gamma_m^{(2)}$ when $m\ge K$. Then $I^1(\Gamma_m^K,S_n)\to\vol_{N_2}^1$ (e.g.~one may go through the proof without effective error terms). However, if $K\gg1$, then for some large $m$, we still have $\Gamma_m^K$ is contained in $N_1$, so $I^1(\Gamma_m^K,S_n)$ is very far from being $\vol_{N_2}^1$, even though $\vol(\Gamma_m^K)$ is already very large.
\end{Ex*}

\paragraph{Remarks on finite volume} 
While we focus on unweighted sequences, our method in fact gives equidistribution of intersection points in the compact case for any sequence of weighted sum of closed geodesic submanifolds that equidistribute to the Liouville measure (however the error term cannot generally be chosen to be uniform).

In contrast, as mentioned above, when $M$ has finite volume but is not compact, we can not in general conclude an equidistribution result for intersection points, unless we work with specific sequences. In the case of $k=1$, the sequence $\{\gamma_T\}$ works; but the following example shows that not every equidistributed sequence does. We remark that this example is closely related to the discontinuity of the intersection form on the space of geodesic currents on a non-compact hyperbolic surface of finite area; cf.\ \cite{torkaman2024intersection}.
\begin{Ex*} \label{ex: cusp}
Let $M$ be a finite-area hyperbolic surface with a cusp. We now construct a sequence $\{\eta_n\}$ where each $\eta_n$ is a weighted sum of closed geodesics such that $\eta_n \to \vol^1_M$ but $I^1(\eta_n,\eta_n) \to 0$ as $n \to \infty$. Let $\gamma_n$ be the sum of closed geodesics with length $\leq n$ and $\alpha_n$ a closed geodesic that goes around the cusp $n$ times. Assume that the homotopy types of $\alpha_n$ after removing the part that goes around the cusp $n$ times is equal to $[\alpha_0]$, see Figure \ref{fig: cusp}.
     
\begin{figure}[htp]
    \centering
    \includegraphics{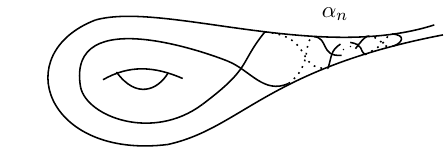}
    \caption{$\alpha_n$, a closed geodesic that goes around the cusp $n$ times}
    \label{fig: cusp}
\end{figure}

Note that the intersection number $i(\alpha_n,\alpha_n)\sim 2n$ (to be compatible with the definition of the intersection number, self-intersection points are counted twice). Now, define $\eta_n=\gamma_n/|\gamma_n|+\alpha_N/|\alpha_{N}|^2$ where $N$ is a large enough integer (with respect to $n$) which goes to infinity as $n \to \infty$. Since $\alpha_N/|\alpha_N|^2 \to 0$, we have $\eta_n \to \vol^1_M$ as $n \to \infty$. Now we estimate the self-intersection measure of $\eta_n$. The behavior of $\alpha_N$ when $N$ is large is similar to $\alpha_0$ union $2r$ where $r$ is a geodesic ray that goes straight into the cusp. Therefore, for large $n$ we have 
\begin{equation}\label{equ: cusp}
I^1(\eta_n,\eta_n)\sim \frac{1}{|I(\eta_n,\eta_n)|} \left( \frac{ I(\gamma_n,\gamma_n)}{|\gamma_n|^2}+\frac{2I(\gamma_n,\alpha_0+2r)}{|\gamma_n||\alpha_N|^2}+\frac{I(\alpha_N,\alpha_N)}{|\alpha_N|^4} \right). 
\end{equation}
Note that we have $|I(\eta_n,\eta_n)|=i(\eta_n,\eta_n)\geq i(\alpha_N,\alpha_N)\geq N$. Therefore, if we pick $N$ large enough relative to $n$, the volume of first two terms in Equation \ref{equ: cusp} tend to zero and the last term is a measure whose support goes to infinity (into the cusp). We conclude the measure $I^1(\eta_n,\eta_n)$ converges to $0$ when $n \to \infty$, as required.
\end{Ex*}

\section{Intersection points in cuspidal neighborhoods}\label{sec:cusp_inter}
In this section, we give upper bounds on the number of intersection points in cuspidal neighborhoods when $k=1$. In the next section, this will be used in the proof of effective equidistribution for $k=1$ to show that there is no escape of mass to the cusp. 

\begin{Thm}\label{thm: cusp_inter_points}
    Let $M$ be a noncompact finite-volume complete hyperbolic manifold of dimension $d \geq 3$. Let $\gamma_T$ be the sum of closed geodesics with length $\leq T$ and $S$ a geodesic hypersurfaces in $M$. There exist positive constants $C, \iota_0$ depending only on the geometry of $M$ so that for any positive $\iota<\iota_0$, the number $|I_\iota(\gamma_T,S)|$ of transverse intersection points between $\gamma_T$ and $S$ in $M_\iota$ satisfies
    $$|I_\iota(\gamma_T,S)|\le C\iota^{d-2}e^{(d-1)T}\vol(S).$$
\end{Thm}
Our method is based on bounding cuspidal excursions of closed geodesics, see Proposition~\ref{prop: geod_excursion}. We believe that an effective control of cuspidal parts of geodesic submanifolds of dimension $k\ge2$ should yield a similar result, although our current method does not seem to give geometric estimates strong enough for this purpose.

\paragraph{Structure of cuspidal neighborhoods}

Recall that given $\iota>0$, the \emph{$\iota$-thin part} of $M$, denoted by $M_{<\iota}$, consists of points $x\in M$ where the injectivity radius $\inj(x)<\iota$. The structure of $M_{<\iota}$ is well-understood when $\iota$ is small enough: there exists a constant $\epsilon_d>0$ depending only on the dimension $d$ (called the Margulis constant), so that when $\iota<\epsilon_d$, each connected component of $M_{<\iota}$ is either a Margulis tube (a small neighborhood of a closed geodesic of length $<\iota$) or a cuspidal neighborhood (see the discussion below). Note that $M_{<\iota}$ contains exactly one cuspidal neighborhood $\mathcal{H}_\iota(\star)$ for each cusp $\star$ of $M$. We use $M_{\iota}$ to denote the union of all cuspidal neighborhoods $\mathcal{H}_\iota(\star)$ in $M_{<\iota}$. For the remainder of the section, we always assume $\iota<\epsilon_d$.

In terms of the upper half space model $\mathbb{H}^d\cong\mathbb{R}^{d-1}\times\mathbb{R}_+$, a cuspidal neighborhood is isomorphic to $\Lambda\backslash\mathbb{H}^d_L$, where $\mathbb{H}^d_L=\{(\xi,\tau)\in\mathbb{R}^{d-1}\times\mathbb{R}: \tau>L\}$ and $\Lambda$ is a discrete subgroup of translations on $\mathbb{R}^{d-1}$ of rank $d-1$. Topologically, a cuspidal neighborhood is homeomorphic to $\mathbb{T}^{d-1}\times(0,\infty)$, where $\mathbb{T}^{d-1}$ is a $(d-1)$-dimensional torus.

Since the hyperbolic metric in this model is given by $ds^2=\dfrac{d\xi^2+d\tau^2}{\tau^2}$, the injectivity radius at the projection of $(\xi,\tau)\in\mathbb{H}^d_L$ to $\Lambda\backslash \mathbb{H}^d_L$ is $c/\tau$ for some constant $c$ depending only on $\Lambda$.

To make consistent and convenient choices throughout the section, we normalize the group $\Lambda$ so that $\mathcal{H}_{\epsilon_d}(\star)\cong\Lambda\backslash\mathbb{H}^d_1$. Then for any $\iota<\epsilon_d$, $\mathcal{H}_{\iota}(\star)\cong\Lambda\backslash\mathbb{H}^d_{\epsilon_d/\iota}$.

\paragraph{Caps, tubes, excursions}
Let $S$ be a geodesic $k$-submanifold in $M$, where $2\le k\le d-1$. Then each connected component of $S\cap\overline{\mathcal{H}_\iota(\star)}$ is either compact, or (the closure of) a cuspidal neighborhood in $S$. We call the former a \emph{cap} and the latter a \emph{tube}. The case $k=1$ is somewhat different: a closed geodesic $\gamma$ only meets $\overline{\mathcal{H}_\iota(\star)}$ in compact segments. We call each segment a \emph{(cuspidal) excursion} of $\gamma$.

We briefly give a geometric picture. Suppose $\mathcal{H}_\iota(\star)\cong\Lambda\backslash\mathbb{H}^d_L$. Lifts of $S$ are totally geodesic $k$-planes in $\mathbb{H}^d$, so are either half $k$-planes or half $k$-spheres in the upper half space model. A cap of $S$ is then given by the projection of the portion above $\tau=L$ of a half $k$-sphere lift of $S$ (which is a spherical cap), and a tube is the corresponding projection of a half $k$-plane lift. On the other hand, lifts of closed geodesics can only be half circles when $\infty$ is a cusp, so an excursion is the projection of the half circle above $\tau=L$.

\paragraph{Geometric estimates}
Let $R$ be a positive integer. Under the normalization $\mathcal{H}_{\epsilon_d}(\star)\cong\Lambda\backslash\mathbb{H}^d_1$, a cap in $S$ is called an \emph{$R$-cap} (resp.\ an excursion in $\gamma$ is called an \emph{$R$-excursion}) if the corresponding lift is a half $k$-sphere (resp.\ a half circle) of Euclidean radius in $[R,R+1)$. The following lemma can be proved similarly to \cite[Lemma~4.4]{torkaman2024intersection}. Here $f\simeq_C g$ means $|f-g|\le C$.
\begin{Lemma}\label{lem: vol.cap}
    Suppose $R\ge 2$. For any $k\ge2$, there exists a constant $C$ depending only on the geometry of $M$ so that any $k$-dimensional $R$-cap is a hyperbolic disk of radius $\simeq_C\log R$, and hence has volume $\asymp R^{k-1}$. Moreover for $k=1$, an $R$-excursion is a geodesic segment of length $\simeq_C2\log R$.
\end{Lemma}

Fix a set of $\mathbb{Z}$-basis for $\Lambda$, i.e. $\Lambda=\mathbb{Z}x_1\oplus\cdots\oplus\mathbb{Z}x_{d-1}$. With respect to this choice of basis, a fundamental domain for the action of $\Lambda$ on $\mathbb{R}^{d-1}$ may be chosen as $F_{\Lambda}:=\{\sum r_ix_i: r_i\in[0,1]\}$. Let $\langle \cdot,\cdot\rangle$ be the inner product associated with this set of basis (i.e.\ $\langle x_i,x_j\rangle=\delta_{ij}$ for $1\le i,j\le d$).

Given an integral vector $u=(u_1,\ldots,u_{d-1})\in\mathbb{Z}^{d-1}$, a $(d-1)$-dimensional tube is called a $u$-tube if its lift intersects the plane $\tau=1$ in a subspace orthogonal to $u_1x_1+\cdots+u_{d-1}x_{d-1}$ with respect to $\langle\cdot,\cdot\rangle$. We have
\begin{Lemma}\label{lem: vol.u}
 Suppose $u$ is a primitive integral vector. There exists a positive constant $C$ depending only on $\Lambda$ and the choice of $\mathbb{Z}$-basis so that the following holds. The orbit of a $u$-tube under the action of $\Lambda$ meets a fixed fundamental domain of $\Lambda$ in at most $C\|u\|$ pieces. Moreover, any $u$-tube has volume $\asymp\|u\|$.
\end{Lemma}
\begin{proof}
    For simplicity, we assume $\Lambda=\mathbb{Z}^n$ with the standard basis $x_i=e_i$. The general case follows by applying an appropriate element of $\GL(d-1,\mathbb{R})$.

Let $u = (u_1, \ldots, u_{d-1})$ be a primitive integral vector. Then the orbits of any hyperplane orthogonal to $u$ under the action of  $\Lambda = \mathbb{Z}^{d-1}$ form a family of equidistant, parallel hyperplanes, spaced $c_{u}$ apart. We claim that $c_{u}=1/\|u\|$. To see this, consider the hyperplane $P$ through the origin orthogonal to $u$. Then the shortest distance from an integer point not lying on $P$ to the plane $P$ is $c_{u}$. Let $a = (a_1, \ldots, a_{d-1}) \in \mathbb{Z}^{d-1}$ be any such point, and let  $b$ denote its orthogonal projection onto $P$. Then $b-a=t\cdot u$ for some scalar $t$, and since $b-a$ is orthogonal to $P$, we have $(b-a)\cdot b = 0.$ Thus we have $t (u \cdot a + t \|u\|^2) = 0$. Solving for $t$, we get:
$$
t = -\frac{u \cdot a}{\|u\|^2}.
$$
Therefore, the distance from $a$ to $P$ is $\|b - a\| = |u \cdot a|/\|u\|.$ Since $u$ is primitive, there exists $a \in \mathbb{Z}^{d-1}$ such that $u \cdot a = \pm 1$, and hence the minimal nonzero distance is $1/\|u\|$, as claimed.
    
    It is then easy to see that the number of such planes intersecting a fundamental domain $F_\Lambda$ is $\asymp C\|u\|$. A definite portion of these intersections have area bounded below by a constant independent of $u$, so it follows that the projection of the plane in $\Lambda\backslash\mathbb{R}^{d-1}$ is a $(d-2)$-dimensional torus of volume $\asymp \|u\|$.
\end{proof}

\paragraph{Bounding geodesic excursions}
Our first goal is to prove the following upper bound on the number of geodesic excursions. Our proof is adapted from \cite{torkaman2024intersection}; see Proposition 4.5 there. Let $A_R^\star(T)$ be the number of $R$-excursions in $\mathcal{H}_{\epsilon_d}(\star)$ among all closed geodesics of length $\le T$. Let $\Lambda'$ be the lattice in the hyperplane $\tau=\epsilon_d$ corresponding to $\PH$, the boundary of $\HS$, induced by $\Lambda$. 

\begin{Prop}\label{prop: geod_excursion}
    There exists a constant $C$ depending only on the geometry of $M$ so that
    $$A_R^\star(T)\le \frac{Ce^{(d-1)T}}{R^d}$$
\end{Prop}

Let $Q\in M$. We first consider the set $\mathcal{G}_T(Q)$ of geodesics based at $Q$ (i.e.~a geodesic segment starting and ending at $Q$) of length $\le T$. Recall that $\inj(Q)$ denotes the injectivity radius at $Q$.
\begin{Lemma}
    There exists a constant $C>0$ depending only on the dimension $d$ so that $$|\mathcal{G}_T(Q)|\le C\cdot\iota(Q)^{-d} e^{(d-1)T},$$
    where $\iota(Q)=\min\{\inj(Q),1\}$.
\end{Lemma}
\begin{proof}
    Arguing exactly as \cite[Lemma~4.2]{torkaman2024intersection}, we conclude that
    $$|\mathcal{G}_T(Q)|\le \vol_d(T+2\iota)/\vol_d(\iota),$$
    for some $\iota<\inj(Q)$. Since $\vol_d(x)\asymp e^{(d-1)x}$ when $x$ is large and $\asymp x^{d}$ when $x$ is small, we have the desired estimate.
\end{proof}

\begin{proof}[Proof of Proposition~\ref{prop: geod_excursion}]
    Fix $Q$ on $\PH$. Let $c$ be the length of $\PH$. Similarly as in \cite{torkaman2024intersection}, we construct a map from $A_R^\star(T)$ to $\mathcal{G}_{T-2\log R+c}$, as follows. Given a closed geodesic $\gamma$ and an $R$-excursion $\alpha$ of it, where $\alpha$ is in $\mathcal{H}_{\epsilon_d}(\star)$ and its endpoints $a_1,a_2$ are on $\PH$, remove $\alpha$ from $\gamma$ and connect $a_1,a_2$ to $Q$ by simple geodesic arcs of total length $\leq c$. See Figure \ref{fig: exc-cut}. Then we obtain a closed curve from $Q$ to itself with length $\leq T-2\log R+c$; the geodesic representative of this closed curve is in $\mathcal{G}_{T-2\log R+c}$.

  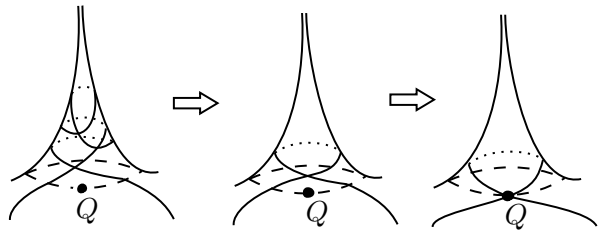
\begin{figure}[h]
      \centering

\tikzset{every picture/.style={line width=0.75pt}} 

\begin{tikzpicture}[x=0.65pt,y=0.55pt,yscale=-0.6,xscale=0.6]

\draw    (28,238) .. controls (92,208) and (98,79) .. (92,22) ;
\draw    (98,23) .. controls (93,60) and (115,241) .. (171,228) ;
\draw  [dash pattern={on 4.5pt off 4.5pt}]  (42,230) .. controls (68,248) and (115,257) .. (155,227) ;
\draw  [dash pattern={on 4.5pt off 4.5pt}]  (42,230) .. controls (82,210) and (120,210) .. (155,227) ;
\draw [line width=0.75]    (66,240) .. controls (95,221) and (119,201) .. (120,178) ;
\draw  [dash pattern={on 0.84pt off 2.51pt}]  (76,176) .. controls (87,162) and (114,161) .. (120,178) ;
\draw [line width=0.75]    (76,176) .. controls (97,198) and (114,172) .. (109,134) ;
\draw [line width=0.75]    (67,200) .. controls (70,218) and (102,231) .. (137,236) ;
\draw  [dash pattern={on 0.84pt off 2.51pt}]  (67,200) .. controls (85,177) and (117,193) .. (128,192) ;
\draw [line width=0.75]    (87,138) .. controls (83,167) and (98,221) .. (128,192) ;
\draw  [dash pattern={on 0.84pt off 2.51pt}]  (87,138) .. controls (95,126) and (103,130) .. (109,134) ;
\draw   (186,143) -- (211.44,143) -- (211.44,137) -- (228.4,149) -- (211.44,161) -- (211.44,155) -- (186,155) -- cycle ;
\draw [line width=0.75]    (27,283) .. controls (26,269) and (44,251) .. (66,240) ;
\draw [line width=0.75]    (137,236) .. controls (168,252) and (179,266) .. (186,282) ;
\draw    (245,244) .. controls (309,214) and (315,85) .. (309,28) ;
\draw    (315,29) .. controls (310,66) and (332,247) .. (388,234) ;
\draw  [dash pattern={on 4.5pt off 4.5pt}]  (259,236) .. controls (285,254) and (332,263) .. (372,233) ;
\draw  [dash pattern={on 4.5pt off 4.5pt}]  (259,236) .. controls (299,216) and (337,216) .. (372,233) ;
\draw [line width=0.75]    (283,246) .. controls (312,227) and (342,236) .. (348,205) ;
\draw [line width=0.75]    (284,206) .. controls (287,224) and (319,237) .. (354,242) ;
\draw  [dash pattern={on 0.84pt off 2.51pt}]  (284,206) .. controls (306,187) and (340,195) .. (348,205) ;
\draw [line width=0.75]    (244,289) .. controls (243,275) and (261,257) .. (283,246) ;
\draw [line width=0.75]    (354,242) .. controls (385,258) and (396,272) .. (403,288) ;
\draw  [color={rgb, 255:red, 255; green, 255; blue, 255 }  ,draw opacity=1 ][line width=3] [line join = round][line cap = round] (82,22) .. controls (84.68,22) and (112.66,19.68) .. (107,31) .. controls (106.96,31.09) and (100.26,31.95) .. (100,32) .. controls (93.78,33.24) and (87.33,33.37) .. (81,33) .. controls (80.26,32.96) and (79.75,32) .. (79,32) ;
\draw  [color={rgb, 255:red, 255; green, 255; blue, 255 }  ,draw opacity=1 ][line width=3] [line join = round][line cap = round] (95,31) .. controls (101.67,31) and (108.34,31.33) .. (115,31) .. controls (115.32,30.98) and (108.42,27.47) .. (107,27) .. controls (103.97,25.99) and (91.62,24.34) .. (87,25) .. controls (86.09,25.13) and (90.31,26.83) .. (91,27) .. controls (95.41,28.1) and (98.66,28) .. (104,28) .. controls (106,28) and (99.98,27.72) .. (98,28) .. controls (95.96,28.29) and (90,27.13) .. (90,30) ;
\draw  [color={rgb, 255:red, 255; green, 255; blue, 255 }  ,draw opacity=1 ][line width=3] [line join = round][line cap = round] (346,34) .. controls (335.23,34) and (328.7,36) .. (320,36) ;
\draw  [color={rgb, 255:red, 255; green, 255; blue, 255 }  ,draw opacity=1 ][line width=3] [line join = round][line cap = round] (329,32) .. controls (334.01,32) and (339.02,31.55) .. (344,31) .. controls (346.98,30.67) and (350,31) .. (353,31) .. controls (354.8,31) and (349.79,29.09) .. (348,29) .. controls (339.67,28.56) and (331.34,28) .. (323,28) ;
\draw  [color={rgb, 255:red, 255; green, 255; blue, 255 }  ,draw opacity=1 ][line width=3] [line join = round][line cap = round] (20,30) .. controls (20,89) and (19.67,148) .. (20,207) .. controls (20.03,212.5) and (22.3,220.91) .. (24,226) .. controls (24.06,226.19) and (25,229) .. (25,229) .. controls (25,229) and (25,227.67) .. (25,227) ;
\draw    (438,246) .. controls (502,216) and (508,87) .. (502,30) ;
\draw    (508,31) .. controls (503,68) and (525,249) .. (581,236) ;
\draw  [dash pattern={on 4.5pt off 4.5pt}]  (452,238) .. controls (478,256) and (525,265) .. (565,235) ;
\draw  [dash pattern={on 4.5pt off 4.5pt}]  (452,238) .. controls (492,218) and (530,218) .. (565,235) ;
\draw [line width=0.75]    (510,256) .. controls (536,245) and (540,232) .. (544,214) ;
\draw [line width=0.75]    (472,219) .. controls (476,236) and (486,244) .. (510,256) ;
\draw  [dash pattern={on 0.84pt off 2.51pt}]  (472,219) .. controls (494,197) and (536,204) .. (544,214) ;
\draw [line width=0.75]    (437,291) .. controls (436,277) and (488,267) .. (510,256) ;
\draw [line width=0.75]    (510,256) .. controls (544,265) and (589,274) .. (596,290) ;
\draw  [color={rgb, 255:red, 255; green, 255; blue, 255 }  ,draw opacity=1 ][line width=3] [line join = round][line cap = round] (95,247) .. controls (88.2,247) and (96,247) .. (96,247) .. controls (97.76,248.17) and (94,251.08) .. (94,248) ;
\draw  [color={rgb, 255:red, 0; green, 0; blue, 0 }  ,draw opacity=1 ][line width=3] [line join = round][line cap = round] (98,247) .. controls (98,244.69) and (96,245.72) .. (96,248) ;
\draw  [color={rgb, 255:red, 0; green, 0; blue, 0 }  ,draw opacity=1 ][line width=3] [line join = round][line cap = round] (315,251) .. controls (315,247.76) and (322.86,253) .. (315,253) ;
\draw  [color={rgb, 255:red, 0; green, 0; blue, 0 }  ,draw opacity=1 ][line width=3] [line join = round][line cap = round] (510,256) .. controls (515.04,256) and (510.81,253.09) .. (509,254) .. controls (507.82,254.59) and (506.83,257) .. (510,257) ;
\draw  [color={rgb, 255:red, 255; green, 255; blue, 255 }  ,draw opacity=1 ][line width=3] [line join = round][line cap = round] (307,41) .. controls (314.67,41) and (322.34,41.33) .. (330,41) .. controls (331.2,40.95) and (329,38.67) .. (328,38) .. controls (326.96,37.3) and (324.42,35.35) .. (323,35) .. controls (316.21,33.3) and (310.46,34) .. (302,34) .. controls (300.67,34) and (304.72,33.63) .. (306,34) .. controls (307.87,34.53) and (309.07,36.79) .. (311,37) .. controls (313.98,37.33) and (317.01,36.7) .. (320,37) .. controls (320.47,37.05) and (321.47,37.96) .. (321,38) .. controls (317.61,38.31) and (306.32,38.21) .. (303,36) .. controls (300.91,34.61) and (312.37,29.05) .. (315,28) .. controls (316.28,27.49) and (320.37,27) .. (319,27) .. controls (299.68,27) and (309.62,30) .. (319,30) ;
\draw  [color={rgb, 255:red, 255; green, 255; blue, 255 }  ,draw opacity=1 ][line width=3] [line join = round][line cap = round] (491,43) .. controls (499.33,43) and (507.67,43) .. (516,43) .. controls (517.8,43) and (512.67,41.67) .. (511,41) .. controls (507.14,39.45) and (501.97,37.99) .. (498,37) .. controls (496.67,36.67) and (493.03,36.97) .. (494,36) .. controls (496.53,33.47) and (508.72,32.54) .. (513,32) .. controls (515.01,31.75) and (517.19,31.91) .. (519,31) .. controls (519.3,30.85) and (520.33,31) .. (520,31) .. controls (513.99,31) and (506.25,27.75) .. (502,32) .. controls (501.47,32.53) and (514.29,35.57) .. (516,36) .. controls (517.65,36.41) and (522.7,37) .. (521,37) .. controls (516.17,37) and (511.51,36.29) .. (507,35) .. controls (504.67,34.33) and (502.17,34.09) .. (500,33) .. controls (499.58,32.79) and (498.53,32) .. (499,32) .. controls (503.9,32) and (520.54,38.19) .. (510,39) .. controls (506.01,39.31) and (502,39) .. (498,39) ;
\draw  [color={rgb, 255:red, 255; green, 255; blue, 255 }  ,draw opacity=1 ][line width=3] [line join = round][line cap = round] (97,30) .. controls (93.85,30) and (91.75,32) .. (89,32) ;
\draw   (396,139) -- (421.44,139) -- (421.44,133) -- (438.4,145) -- (421.44,157) -- (421.44,151) -- (396,151) -- cycle ;


\draw (88,254.4) node [anchor=north west][inner sep=0.75pt]    {$Q$};
\draw (311,258.4) node [anchor=north west][inner sep=0.75pt]    {$Q$};
\draw (504,261.4) node [anchor=north west][inner sep=0.75pt]    {$Q$};

\end{tikzpicture}
        \caption{How the map modifies a curve by removing an $R-$excursion}
        \label{fig: exc-cut}
    \end{figure}

    We can see that this map is at most $c'R^{d-2}$ to $1$ since there are at most $c'R^{d-2}$ class of $R$-excursions in $\mathcal{H}_{\epsilon_d}(\star)$ up to homotopy with endpoints stay in a ball of radius $\leq 2c$ inside the boundary of $\mathcal{H}_{\epsilon_d}(\star)$. To see that fix a preimage of $a_1$ in the upper half-space, then preimage of $a_2$ is a point between two balls of radii $R-1$ and $R+1$ centered at $a_1$ inside the hyperplane $\tau=\epsilon_d$ (corresponding to $\PH$). We can see that the number of different lattice cubes of $\Lambda$ inside this bounded region is $\leq c'R^{d-2}$; see Figure \ref{fig:circle-boundary}.

\begin{figure}[htp]
    \centering

\tikzset{every picture/.style={line width=0.75pt}} 

\begin{tikzpicture}[x=0.75pt,y=0.75pt,yscale=-0.7,xscale=0.7]

\draw  [draw opacity=0] (102,40.2) -- (374.44,40.2) -- (374.44,272.2) -- (102,272.2) -- cycle ; \draw   (102,40.2) -- (102,272.2)(122,40.2) -- (122,272.2)(142,40.2) -- (142,272.2)(162,40.2) -- (162,272.2)(182,40.2) -- (182,272.2)(202,40.2) -- (202,272.2)(222,40.2) -- (222,272.2)(242,40.2) -- (242,272.2)(262,40.2) -- (262,272.2)(282,40.2) -- (282,272.2)(302,40.2) -- (302,272.2)(322,40.2) -- (322,272.2)(342,40.2) -- (342,272.2)(362,40.2) -- (362,272.2) ; \draw   (102,40.2) -- (374.44,40.2)(102,60.2) -- (374.44,60.2)(102,80.2) -- (374.44,80.2)(102,100.2) -- (374.44,100.2)(102,120.2) -- (374.44,120.2)(102,140.2) -- (374.44,140.2)(102,160.2) -- (374.44,160.2)(102,180.2) -- (374.44,180.2)(102,200.2) -- (374.44,200.2)(102,220.2) -- (374.44,220.2)(102,240.2) -- (374.44,240.2)(102,260.2) -- (374.44,260.2) ; \draw    ;
\draw   (169.5,151.1) .. controls (169.5,115.59) and (198.29,86.8) .. (233.8,86.8) .. controls (269.31,86.8) and (298.1,115.59) .. (298.1,151.1) .. controls (298.1,186.61) and (269.31,215.4) .. (233.8,215.4) .. controls (198.29,215.4) and (169.5,186.61) .. (169.5,151.1) -- cycle ;
\draw  [line width=3] [line join = round][line cap = round] (232.22,152.1) .. controls (232.22,151.76) and (232.22,151.43) .. (232.22,151.1) ;
\draw  [fill={rgb, 255:red, 155; green, 155; blue, 155 }  ,fill opacity=1 ] (162,119.6) -- (181.4,119.6) -- (181.4,139) -- (162,139) -- cycle ;
\draw  [fill={rgb, 255:red, 155; green, 155; blue, 155 }  ,fill opacity=1 ] (162,139.6) -- (181.4,139.6) -- (181.4,159) -- (162,159) -- cycle ;
\draw  [fill={rgb, 255:red, 155; green, 155; blue, 155 }  ,fill opacity=1 ] (162,159) -- (181.4,159) -- (181.4,178.4) -- (162,178.4) -- cycle ;
\draw  [fill={rgb, 255:red, 155; green, 155; blue, 155 }  ,fill opacity=1 ] (162,179.6) -- (181.4,179.6) -- (181.4,199) -- (162,199) -- cycle ;
\draw  [fill={rgb, 255:red, 155; green, 155; blue, 155 }  ,fill opacity=1 ] (182.6,179) -- (202,179) -- (202,198.4) -- (182.6,198.4) -- cycle ;
\draw  [fill={rgb, 255:red, 155; green, 155; blue, 155 }  ,fill opacity=1 ] (182,199.6) -- (201.4,199.6) -- (201.4,219) -- (182,219) -- cycle ;
\draw  [fill={rgb, 255:red, 155; green, 155; blue, 155 }  ,fill opacity=1 ] (202.6,200.2) -- (222,200.2) -- (222,219.6) -- (202.6,219.6) -- cycle ;
\draw  [fill={rgb, 255:red, 155; green, 155; blue, 155 }  ,fill opacity=1 ] (222,199.6) -- (241.4,199.6) -- (241.4,219) -- (222,219) -- cycle ;
\draw  [fill={rgb, 255:red, 155; green, 155; blue, 155 }  ,fill opacity=1 ] (242.6,199.6) -- (262,199.6) -- (262,219) -- (242.6,219) -- cycle ;
\draw  [fill={rgb, 255:red, 155; green, 155; blue, 155 }  ,fill opacity=1 ] (262,199.6) -- (281.4,199.6) -- (281.4,219) -- (262,219) -- cycle ;
\draw  [fill={rgb, 255:red, 155; green, 155; blue, 155 }  ,fill opacity=1 ] (262.6,179.6) -- (282,179.6) -- (282,199) -- (262.6,199) -- cycle ;
\draw  [fill={rgb, 255:red, 155; green, 155; blue, 155 }  ,fill opacity=1 ] (282.6,179.6) -- (302,179.6) -- (302,199) -- (282.6,199) -- cycle ;
\draw  [fill={rgb, 255:red, 155; green, 155; blue, 155 }  ,fill opacity=1 ] (282,159.6) -- (301.4,159.6) -- (301.4,179) -- (282,179) -- cycle ;
\draw  [fill={rgb, 255:red, 155; green, 155; blue, 155 }  ,fill opacity=1 ] (282.6,139) -- (302,139) -- (302,158.4) -- (282.6,158.4) -- cycle ;
\draw  [fill={rgb, 255:red, 155; green, 155; blue, 155 }  ,fill opacity=1 ] (282,119.6) -- (301.4,119.6) -- (301.4,139) -- (282,139) -- cycle ;
\draw  [fill={rgb, 255:red, 155; green, 155; blue, 155 }  ,fill opacity=1 ] (282,99) -- (301.4,99) -- (301.4,118.4) -- (282,118.4) -- cycle ;
\draw  [fill={rgb, 255:red, 155; green, 155; blue, 155 }  ,fill opacity=1 ] (262,99.6) -- (281.4,99.6) -- (281.4,119) -- (262,119) -- cycle ;
\draw  [fill={rgb, 255:red, 155; green, 155; blue, 155 }  ,fill opacity=1 ] (262,79.6) -- (281.4,79.6) -- (281.4,99) -- (262,99) -- cycle ;
\draw  [fill={rgb, 255:red, 155; green, 155; blue, 155 }  ,fill opacity=1 ] (242.6,80.2) -- (262,80.2) -- (262,99.6) -- (242.6,99.6) -- cycle ;
\draw  [fill={rgb, 255:red, 155; green, 155; blue, 155 }  ,fill opacity=1 ] (222,79.6) -- (241.4,79.6) -- (241.4,99) -- (222,99) -- cycle ;
\draw  [fill={rgb, 255:red, 155; green, 155; blue, 155 }  ,fill opacity=1 ] (202.6,79.6) -- (222,79.6) -- (222,99) -- (202.6,99) -- cycle ;
\draw  [fill={rgb, 255:red, 155; green, 155; blue, 155 }  ,fill opacity=1 ] (182,79.6) -- (201.4,79.6) -- (201.4,99) -- (182,99) -- cycle ;
\draw  [fill={rgb, 255:red, 155; green, 155; blue, 155 }  ,fill opacity=1 ] (182.6,99.6) -- (202,99.6) -- (202,119) -- (182.6,119) -- cycle ;
\draw  [fill={rgb, 255:red, 155; green, 155; blue, 155 }  ,fill opacity=1 ] (162,99) -- (181.4,99) -- (181.4,118.4) -- (162,118.4) -- cycle ;
\draw  [fill={rgb, 255:red, 155; green, 155; blue, 155 }  ,fill opacity=1 ] (162.6,199.6) -- (182,199.6) -- (182,219) -- (162.6,219) -- cycle ;
\draw  [fill={rgb, 255:red, 155; green, 155; blue, 155 }  ,fill opacity=1 ] (202.6,220.2) -- (222,220.2) -- (222,239.6) -- (202.6,239.6) -- cycle ;
\draw  [fill={rgb, 255:red, 155; green, 155; blue, 155 }  ,fill opacity=1 ] (222,220.2) -- (241.4,220.2) -- (241.4,239.6) -- (222,239.6) -- cycle ;
\draw  [fill={rgb, 255:red, 155; green, 155; blue, 155 }  ,fill opacity=1 ] (242.6,220.2) -- (262,220.2) -- (262,239.6) -- (242.6,239.6) -- cycle ;
\draw  [fill={rgb, 255:red, 155; green, 155; blue, 155 }  ,fill opacity=1 ] (282.6,200.8) -- (302,200.8) -- (302,220.2) -- (282.6,220.2) -- cycle ;
\draw  [fill={rgb, 255:red, 155; green, 155; blue, 155 }  ,fill opacity=1 ] (302,160.2) -- (321.4,160.2) -- (321.4,179.6) -- (302,179.6) -- cycle ;
\draw  [fill={rgb, 255:red, 155; green, 155; blue, 155 }  ,fill opacity=1 ] (302.6,139.6) -- (322,139.6) -- (322,159) -- (302.6,159) -- cycle ;
\draw  [fill={rgb, 255:red, 155; green, 155; blue, 155 }  ,fill opacity=1 ] (302.6,120.2) -- (322,120.2) -- (322,139.6) -- (302.6,139.6) -- cycle ;
\draw  [fill={rgb, 255:red, 155; green, 155; blue, 155 }  ,fill opacity=1 ] (281.4,80.2) -- (300.8,80.2) -- (300.8,99.6) -- (281.4,99.6) -- cycle ;
\draw  [fill={rgb, 255:red, 155; green, 155; blue, 155 }  ,fill opacity=1 ] (142,139.6) -- (161.4,139.6) -- (161.4,159) -- (142,159) -- cycle ;
\draw  [color={rgb, 255:red, 208; green, 2; blue, 27 }  ,draw opacity=1 ] (174.27,151.1) .. controls (174.27,118.22) and (200.92,91.57) .. (233.8,91.57) .. controls (266.68,91.57) and (293.33,118.22) .. (293.33,151.1) .. controls (293.33,183.98) and (266.68,210.63) .. (233.8,210.63) .. controls (200.92,210.63) and (174.27,183.98) .. (174.27,151.1) -- cycle ;
\draw  [color={rgb, 255:red, 208; green, 2; blue, 27 }  ,draw opacity=1 ] (136.91,151.1) .. controls (136.91,97.59) and (180.29,54.21) .. (233.8,54.21) .. controls (287.31,54.21) and (330.69,97.59) .. (330.69,151.1) .. controls (330.69,204.61) and (287.31,247.99) .. (233.8,247.99) .. controls (180.29,247.99) and (136.91,204.61) .. (136.91,151.1) -- cycle ;
\draw  [fill={rgb, 255:red, 155; green, 155; blue, 155 }  ,fill opacity=1 ] (202,239.6) -- (221.4,239.6) -- (221.4,259) -- (202,259) -- cycle ;
\draw  [fill={rgb, 255:red, 155; green, 155; blue, 155 }  ,fill opacity=1 ] (262,220.2) -- (281.4,220.2) -- (281.4,239.6) -- (262,239.6) -- cycle ;
\draw  [fill={rgb, 255:red, 155; green, 155; blue, 155 }  ,fill opacity=1 ] (182.6,220.2) -- (202,220.2) -- (202,239.6) -- (182.6,239.6) -- cycle ;
\draw  [fill={rgb, 255:red, 155; green, 155; blue, 155 }  ,fill opacity=1 ] (163.2,220.2) -- (182.6,220.2) -- (182.6,239.6) -- (163.2,239.6) -- cycle ;
\draw  [fill={rgb, 255:red, 155; green, 155; blue, 155 }  ,fill opacity=1 ] (282.6,220.2) -- (302,220.2) -- (302,239.6) -- (282.6,239.6) -- cycle ;
\draw  [fill={rgb, 255:red, 155; green, 155; blue, 155 }  ,fill opacity=1 ] (222.6,240.8) -- (242,240.8) -- (242,260.2) -- (222.6,260.2) -- cycle ;
\draw  [fill={rgb, 255:red, 155; green, 155; blue, 155 }  ,fill opacity=1 ] (242.6,240.2) -- (262,240.2) -- (262,259.6) -- (242.6,259.6) -- cycle ;
\draw  [fill={rgb, 255:red, 155; green, 155; blue, 155 }  ,fill opacity=1 ] (262,240.2) -- (281.4,240.2) -- (281.4,259.6) -- (262,259.6) -- cycle ;
\draw  [fill={rgb, 255:red, 155; green, 155; blue, 155 }  ,fill opacity=1 ] (302.6,200.8) -- (322,200.8) -- (322,220.2) -- (302.6,220.2) -- cycle ;
\draw  [fill={rgb, 255:red, 155; green, 155; blue, 155 }  ,fill opacity=1 ] (302,181.4) -- (321.4,181.4) -- (321.4,200.8) -- (302,200.8) -- cycle ;
\draw  [fill={rgb, 255:red, 155; green, 155; blue, 155 }  ,fill opacity=1 ] (322.6,180.8) -- (342,180.8) -- (342,200.2) -- (322.6,200.2) -- cycle ;
\draw  [fill={rgb, 255:red, 155; green, 155; blue, 155 }  ,fill opacity=1 ] (322.6,160.2) -- (342,160.2) -- (342,179.6) -- (322.6,179.6) -- cycle ;
\draw  [fill={rgb, 255:red, 155; green, 155; blue, 155 }  ,fill opacity=1 ] (322.6,140.8) -- (342,140.8) -- (342,160.2) -- (322.6,160.2) -- cycle ;
\draw  [fill={rgb, 255:red, 155; green, 155; blue, 155 }  ,fill opacity=1 ] (322,120.8) -- (341.4,120.8) -- (341.4,140.2) -- (322,140.2) -- cycle ;
\draw  [fill={rgb, 255:red, 155; green, 155; blue, 155 }  ,fill opacity=1 ] (142.6,200.8) -- (162,200.8) -- (162,220.2) -- (142.6,220.2) -- cycle ;
\draw  [fill={rgb, 255:red, 155; green, 155; blue, 155 }  ,fill opacity=1 ] (142,180.8) -- (161.4,180.8) -- (161.4,200.2) -- (142,200.2) -- cycle ;
\draw  [fill={rgb, 255:red, 155; green, 155; blue, 155 }  ,fill opacity=1 ] (122.6,160.8) -- (142,160.8) -- (142,180.2) -- (122.6,180.2) -- cycle ;
\draw  [fill={rgb, 255:red, 155; green, 155; blue, 155 }  ,fill opacity=1 ] (142,161.4) -- (161.4,161.4) -- (161.4,180.8) -- (142,180.8) -- cycle ;
\draw  [fill={rgb, 255:red, 155; green, 155; blue, 155 }  ,fill opacity=1 ] (122.6,140.8) -- (142,140.8) -- (142,160.2) -- (122.6,160.2) -- cycle ;
\draw  [fill={rgb, 255:red, 155; green, 155; blue, 155 }  ,fill opacity=1 ] (302.6,100.2) -- (322,100.2) -- (322,119.6) -- (302.6,119.6) -- cycle ;
\draw  [fill={rgb, 255:red, 155; green, 155; blue, 155 }  ,fill opacity=1 ] (322,101.4) -- (341.4,101.4) -- (341.4,120.8) -- (322,120.8) -- cycle ;
\draw  [fill={rgb, 255:red, 155; green, 155; blue, 155 }  ,fill opacity=1 ] (301.4,79.6) -- (320.8,79.6) -- (320.8,99) -- (301.4,99) -- cycle ;
\draw  [fill={rgb, 255:red, 155; green, 155; blue, 155 }  ,fill opacity=1 ] (281.4,60.8) -- (300.8,60.8) -- (300.8,80.2) -- (281.4,80.2) -- cycle ;
\draw  [fill={rgb, 255:red, 155; green, 155; blue, 155 }  ,fill opacity=1 ] (122.6,120.2) -- (142,120.2) -- (142,139.6) -- (122.6,139.6) -- cycle ;
\draw  [fill={rgb, 255:red, 155; green, 155; blue, 155 }  ,fill opacity=1 ] (142,121.4) -- (161.4,121.4) -- (161.4,140.8) -- (142,140.8) -- cycle ;
\draw  [fill={rgb, 255:red, 155; green, 155; blue, 155 }  ,fill opacity=1 ] (142.6,100.2) -- (162,100.2) -- (162,119.6) -- (142.6,119.6) -- cycle ;
\draw  [fill={rgb, 255:red, 155; green, 155; blue, 155 }  ,fill opacity=1 ] (142.6,80.2) -- (162,80.2) -- (162,99.6) -- (142.6,99.6) -- cycle ;
\draw  [fill={rgb, 255:red, 155; green, 155; blue, 155 }  ,fill opacity=1 ] (162.6,79.6) -- (182,79.6) -- (182,99) -- (162.6,99) -- cycle ;
\draw  [fill={rgb, 255:red, 155; green, 155; blue, 155 }  ,fill opacity=1 ] (162,60.8) -- (181.4,60.8) -- (181.4,80.2) -- (162,80.2) -- cycle ;
\draw  [fill={rgb, 255:red, 155; green, 155; blue, 155 }  ,fill opacity=1 ] (181.4,60.8) -- (200.8,60.8) -- (200.8,80.2) -- (181.4,80.2) -- cycle ;
\draw  [fill={rgb, 255:red, 155; green, 155; blue, 155 }  ,fill opacity=1 ] (202,60.8) -- (221.4,60.8) -- (221.4,80.2) -- (202,80.2) -- cycle ;
\draw  [fill={rgb, 255:red, 155; green, 155; blue, 155 }  ,fill opacity=1 ] (222.6,60.8) -- (242,60.8) -- (242,80.2) -- (222.6,80.2) -- cycle ;
\draw  [fill={rgb, 255:red, 155; green, 155; blue, 155 }  ,fill opacity=1 ] (182.6,240.8) -- (202,240.8) -- (202,260.2) -- (182.6,260.2) -- cycle ;
\draw  [fill={rgb, 255:red, 155; green, 155; blue, 155 }  ,fill opacity=1 ] (242.6,60.8) -- (262,60.8) -- (262,80.2) -- (242.6,80.2) -- cycle ;
\draw  [fill={rgb, 255:red, 155; green, 155; blue, 155 }  ,fill opacity=1 ] (262.6,60.2) -- (282,60.2) -- (282,79.6) -- (262.6,79.6) -- cycle ;
\draw  [fill={rgb, 255:red, 155; green, 155; blue, 155 }  ,fill opacity=1 ] (202.6,40.8) -- (222,40.8) -- (222,60.2) -- (202.6,60.2) -- cycle ;
\draw  [fill={rgb, 255:red, 155; green, 155; blue, 155 }  ,fill opacity=1 ] (223.2,41.4) -- (242.6,41.4) -- (242.6,60.8) -- (223.2,60.8) -- cycle ;
\draw  [fill={rgb, 255:red, 155; green, 155; blue, 155 }  ,fill opacity=1 ] (242.6,41.4) -- (262,41.4) -- (262,60.8) -- (242.6,60.8) -- cycle ;
\draw  [color={rgb, 255:red, 208; green, 2; blue, 27 }  ,draw opacity=1 ] (136.45,152.57) .. controls (136.45,98.49) and (180.29,54.65) .. (234.38,54.65) .. controls (288.46,54.65) and (332.3,98.49) .. (332.3,152.57) .. controls (332.3,206.66) and (288.46,250.5) .. (234.38,250.5) .. controls (180.29,250.5) and (136.45,206.66) .. (136.45,152.57) -- cycle ;

\draw (226,154) node [anchor=north west][inner sep=0.75pt]    {$a_{1}$};

\end{tikzpicture}

    \caption{ Plane for $\tau=\epsilon_d$ (corresponding to $\PH$) for $d=3$ tiled by fundamental domains of the lattice $\Lambda'$. The other endpoint of the geodesic excursion, $a_2$, lies in one of the shaded fundamental domains.}
    \label{fig:circle-boundary}
\end{figure}
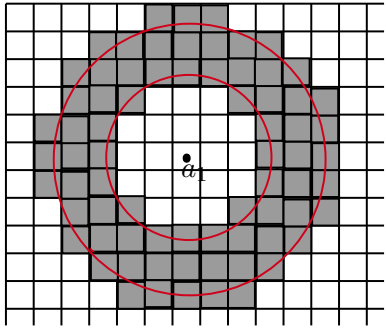

    Hence
    $$|A_R^\star(T)|\le CR^{d-2}e^{(d-1)T}R^{-2(d-1)}=Ce^{(d-1)T}R^{-d},$$
    as desired.
\end{proof}

\paragraph{Bounding intersection points}
Finally we have the following estimates on the number of intersections in a cuspidal neighborhood.

Let $N(R_1,R_2)$ denote the number of the intersection points between an $R_1$-excursion and a $(d-1)$-dimensional $R_2$-cap in $\mathcal{H}_{\epsilon_d}(\star)$.

\begin{Lemma}\label{lem: inter_upper_bound}
    There exists $C, R_0>0$ depending only on $\Lambda$ such that for any integers $R_1, R_2\ge R_0$, we have:
    
    $$N(R_1,R_2)\le \begin{cases}
        CR_1R_2^{d-2}&\text{if $R_1\le R_2$,}\\
        CR_2^{d-1}&\text{if $R_1>R_2$.}
    \end{cases}$$
\end{Lemma}
\begin{proof}
The number of intersections is equal to the number of preimages of the $R_1$-excursion (or $R_2$-cap) that intersect a fixed preimage of the $R_2$-cap (or $R_1$-excursion). Assume that $R_1 \leq R_2$, and fix a preimage of the $R_2$-cap. Its intersection with the hyperplane $\PH$ is a $(d-2)$-dimensional sphere, which we denote by $S$.

A preimage of the $R_1$-excursion intersects $\PH$ at two points $\tilde{a}_1, \tilde{a}_2$, where the vector $\tilde{a}_2 - \tilde{a}_1$ is a fixed vector $v$ of length approximately $2R_1$. The set of such preimages corresponds to the translations of the geodesic segment $\tilde{a}_1 \tilde{a}_2$ by elements of the lattice $\Lambda'$.

A geodesic segment $\tilde{a}_1 \tilde{a}_2$ intersects the fixed  $R_2$-cap if and only if one endpoint lies inside $S$ and the other lies outside. Thus, $\tilde{a}_1$ must lie within the $\|v\|$-neighborhood of $S$. The number of such $\tilde{a}_1$ is approximately the volume of this neighborhood, which is asymptotically $\sim C R_1 R_2^{d-2}$.

In the case $R_1>R_2$, we fix a preimage of the excursion, intersecting $\PH$ at two points $\tilde{a}_1, \tilde{a}_2$. Similar to the previous case, the number of intersecting $R_2$-caps is approximately equal to the volume of the $2R_2$-neighborhood of  $\tilde{a}_1$ and $\tilde{a}_2$ in $\PH$, which is asymptotically $\sim C R_2^{d-1}$.
\end{proof}

Let $Q(R,u)$ denote the number of intersections between an $R$-excursion and a $u$-tube in $\mathcal{H}_{\epsilon_d}(\star)$.

\begin{Lemma}
    There exists $C,R_0>0$ depending only on $\Lambda$ and the choice of integral basis so that for any integers $R\ge R_0$ and any primitive integral vector $u$, we have:
    $$Q(R,u)\le CR\|u\|.$$
\end{Lemma}
\begin{proof}
    Note that orbits of a lift of an $R$-excursion meet a fixed fundamental domain of the cuspidal region at most $\asymp R$ times. Each time it meets at most $\asymp\|u\|$ pieces of the $u$-tube by Lemma \ref{lem: vol.u}. Hence $Q(R,u)\le CR\|u\|$ as desired.
\end{proof}

We are now ready to prove Theorem~\ref{thm: cusp_inter_points}.
\begin{proof}[Proof of Theorem~\ref{thm: cusp_inter_points}]
    Let $A_R(T)$ denote the number of all $R$-excursions among closed geodesics of length $\le T$, $\#_R$ the number of $(d-1)$-dimensional $R$-caps, and $\#_u$ the number of $u$-tubes.
    
    We only have to consider $R$-caps and excursions with $R+1>\epsilon_d/\iota$. Set $N_\iota=\lceil\epsilon_d/\iota\rceil-1$. From the two previous lemma, we have
    \begin{align*}
        &|I_\iota(\gamma_T,S)|\le \sum_{R_1,R_2=N_\iota}^\infty N(R_1,R_2)A_{R_1}(T)\#_{R_2}+\sum \limits_{u}\sum_{R=N_\iota}^\infty Q(R,u)A_{R}(T)\#_u\\
        &\le C\sum_{R_2=N_\iota}^\infty \sum_{R_1=N_\iota}^{R_2}R_1^{-(d-1)}R_2^{d-2}e^{(d-1)T}\#_{R_2}+C\sum_{R_2=N_\iota}^\infty\sum_{R_1=R_2+1}^\infty R_1^{-d}R_2^{d-1}e^{(d-1)T}\#_{R_2}\\
        &\qquad\qquad+C\sum_{u}\sum_{R_1=N_\iota}^\infty R_1^{-(d-1)}\|u\|e^{(d-1)T}\#_u\\
        &\le Ce^{(d-1)T}\left(N_\iota^{-(d-2)}\sum_{R_2=N_\iota}^\infty R_2^{d-2}\#_{R_2}+\sum_{R_2=N_\iota}^\infty\#_{R_2}+N_\iota^{-(d-2)}\sum_{u}\|u\|\#_u\right)
    \end{align*}
    Since the volume of an $R_2$-cap is given by $\asymp R_2^{d-2}$ (see Lemma \ref{lem: vol.cap}), we have 
    
    $$\sum_{R_2=N_\iota}^\infty R_2^{d-2}\#_{R_2} \leq c \vol(S)$$. 
    
    Similarly, since the volume of a $u$-tube is given by $\asymp \|u\|$ (see Lemma \ref{lem: vol.u}), we have 
    $$
    \sum_{u}\|u\|\#_u \leq c\vol(S).
    $$ 
    Therefore, the last expression in the bracket is bounded above by $CN_\iota^{-(d-2)}\vol(S)$. Now $N_\iota\asymp1/\iota$, which leads to the estimate
    $$|I_\iota(\gamma_T,S)|\le Ce^{(d-1)T}\iota^{d-2}\vol(S)$$
    as desired.
\end{proof}
\paragraph{On the case $2\le k\le d-1$}
    Unlike closed geodesics, we cannot apply the ``cut-and-homotope'' strategy in Proposition~\ref{prop: geod_excursion} to modify caps / tubes of a geodesic submanifold of dimension $k\in [2,d-1]$. Given a geodesic $k$-submanifold, let $A_R^\star(S)$ be the number of $R$-caps contained in $S$ inside $\mathcal{H}_{\epsilon_d}(\star)$. Then, based on a volume computation and by Lemma \ref{lem: vol.cap}, we have the easy upper bound
    \begin{equation}
        |A_R^\star(S)|\le C\frac{\vol(S)}{R^{k-1}}\tag{$\dagger$}
    \end{equation}
    which is generally much weaker than the bound in Proposition~\ref{prop: geod_excursion}.

    Without going into details, we remark that with only $(\dagger)$, our method does not yield effective control of the number of intersections in cuspidal neighborhoods. However, if we could show a stronger estimate e.g.\ $|A_R^\star(S)|\le C\dfrac{\vol(S)}{R^{k}}$, then we would have the equidistribution result. 

On the other hand, to prove Theorem \ref{thm: general}, which is weaker than our Theorem \ref{thm: main} (equidistribution for $k=1$), we need the following much weaker estimate in \S\ref{sec:estimates2} for the case $2\le k\le d-2$. Let $\Gamma$ and $S$ be geodesic submanifolds in $M$ of dimension $k$ and $d-k$ respectively. Recall that $M_{\iota}$ is the union of $\iota$-thin neighborhood of the cusps. For any $0<\iota<\epsilon_d$ and $L>1$, consider the region $\mathcal{R}({\iota,L}):=M_{\iota}\backslash M_{\iota/L}$. We denote by $|I_{\mathcal{R}(\iota,L)}(\Gamma,S)|$ the number of transverse intersection points between $\Gamma$ and $S$ in $\mathcal{R}({\iota,L})$, and $|I_{\ge\iota}(\Gamma,S)|$ the number of intersection points in $M_{\ge\iota}=\overline{M\backslash M_{\iota}}$. Then we have
\begin{Lemma}\label{lem: ratio}
    There exists a constant $\iota_0\in(0,\epsilon_d)$ depending only on the geometry of $M$ so that the following holds. For any $0<\iota<\iota_0$, there exist constants $C_\iota,D_{\iota}>0$ (depending only on $\iota$ and the geometry of $M$) so that
    $$\frac{|I_{\mathcal{R}({\iota,L})}(\Gamma,S)|}{|I_{\ge\iota}(\Gamma,S)|}\le D_{\iota}$$
    for any $1<L\leq C_\iota$.
\end{Lemma}
\begin{proof}
    Take $\iota_0=\epsilon_d/4$. For any $\iota<\iota_0$, set $R_\iota:=\epsilon_d/\iota>4$, and $C_\iota=\log(R_\iota+1)-\log R_\iota$. We will show that the number of intersection points in $\mathcal{R}(\iota,L)$ is comparable with that in $M_{\ge\iota}$.

    We will focus on the case of caps; the case involving tubes is similar. Let $\#^k_R$ (resp.\ $\#_R^{d-k}$) be the number of $k$-dimensional (resp.\ $(d-k)$-dimensional) $R$-caps in $\Gamma$ (resp.\ $S$). We use $N_{\iota,C}(R_1,R_2)$ to denote the number of intersection points in $\mathcal{R}(\iota,L)$ between a $k$-dimensional $R_1$-cap and a $(d-k)$-dimensional $R_2$-cap.
    
    The lemma follows from the following sequence of claims.
\begin{ClaimRoman}
    We have the following upper bound
    $$N_{\iota,C_\iota}(R_1,R_2)=\begin{cases}
        O(R_2^{d-k-1})&\text{if }R_1\le R_2,\\O(R_1^{k-1})&\text{if }R_1>R_2.
    \end{cases}$$
    where the implied constant in $O(\cdot)$ depends only on $\iota$ and the geometry of $M$.
\end{ClaimRoman}
The proof of this claim is similar to that of Lemma~\ref{lem: inter_upper_bound}.
\begin{ClaimRoman}
    We have the upper bound
    $$|I_{\mathcal{R}(\iota,C_\iota)}(\Gamma,S)|=O\left(\vol(S)\sum_{R_1\ge R_\iota}\#_{R_1}^k+\vol(\Gamma)\sum_{R_2\ge R_\iota}\#_{R_2}^{d-k}\right).$$
\end{ClaimRoman}
Indeed, from the previous claim, we have
\begin{align*}
    |I_{\mathcal{R}(\iota,C_\iota)}(\Gamma,S)|\le\sum_{R_1,R_2\ge R_\iota}\#^k_{R_1}\#^{d-k}_{R_2}N_{\iota,C_\iota}(R_1,R_2)\\
    \le \sum_{R_{\iota} \le R_1\le R_2}\#^k_{R_1}\#^{d-k}_{R_2}O(R_2^{d-k-1})+\sum_{R_1> R_2 \geq R_{\iota}}\#^k_{R_1}\#^{d-k}_{R_2}O(R_1^{k-1})
\end{align*}
Now since $\sum_{R_1}\#^k_{R_1}R_1^{k-1}=O(\vol(\Gamma))$ and $\sum_{R_2}\#^k_{R_2}R_2^{d-k-1}=O(\vol(S))$, we have the claim.
\begin{ClaimRoman}
    We have
    $$\vol(S)\sum_{R_{\iota} \leq R_1}\#_{R_1}^k+\vol(\Gamma)\sum_{R_{\iota} \leq R_2}\#_{R_2}^{d-k}=O(|I_{\ge\iota}(\Gamma,S)|).$$
\end{ClaimRoman}
To see this, note that any $R$-cap intersects $\PH$ with an angle in a specific range. So by applying the equidistribution result Theorem~\ref{thm: mixing.surface} and assuming $\vol(S)$ and $\vol(\Gamma)$ are large enough, we conclude that $\#_{R_\iota/3}^k\asymp\vol(\Gamma)$ and $\#_{R_\iota/3}^{d-k}\asymp\vol(S)$ (the multiplicative constants depend on $\iota$ and $M$). Now every $R_\iota/3$-cap in $\Gamma$ must intersect a $R_2$-cap in $S$ if $R_2\ge R_{\iota}$, and similarly every $R_\iota/3$-cap in $S$ must intersect a $R_1$-cap in $\Gamma$ if $R_1\ge R_\iota$. These intersection points all lie in $M_{\ge\iota}$. So we have 
\begin{align*}
|I_{\ge\iota}(\Gamma,S)|\ge\sum_{R_1\ge R_\iota}\#_{R_1}^{k}\#_{R_\iota/3}^{d-k}+\sum_{R_2\ge R_\iota}\#_{R_\iota/3}^{k}\#_{R_2}^{d-k}  \\
\asymp \vol(S)\sum_{R_1\geq R_{\iota}}\#_{R_1}^k+\vol(\Gamma)\sum_{R_2 \geq R_{\iota}}\#_{R_2}^{d-k}
\end{align*}
This gives the desired estimates.
\end{proof}

\section{Effective estimates I: the case \texorpdfstring{$k=1$}{k=1}}\label{sec:estimates1}

We are now ready to describe the effective estimates. We first give a complete proof of the case $k=1$, which illustrates the main ideas without additional technicalities. We will discuss the general case in the next section.

Throughout this section, all constants implied in the expression $O(\cdot)$ depend only on the geometry of $M$. After the proof of each claim, we will record the constraints on some of the variables introduced to guarantee convergence.

Given a sum of closed geodesics $\gamma$ and a geodesic hypersurface $S$ on $M$, define
$$
c(\gamma,S):=\frac{\ell(\gamma)\vol(S)}{|I(\gamma,S)|}.
$$
Recall that $\gamma_T$ is sum of the closed geodesics with length $\leq T$.

\paragraph{Step 1: Relate to the intersection kernel}
Our first claim is an effective version of Proposition~\ref{prop: intersection.kernel}.
\begin{Claim}\label{claim: inter_kernel}
For any $f \in C^1(M)$ we have
$$
    \left|\int_{M} f(x) \, dI^1(\gamma_T,S)-c(\gamma_T,S)\int_{\mathcal{F}(M)} f(v) \int_{\mathcal{F}(M)} K_{\delta}(u,v)\, d\mu^1_{\gamma_T}(u)d\mu^1_{S}(v)\right|
    \leq \delta \, \Lip(f).
$$

\end{Claim}

\begin{proof}
    Let $\widetilde S$ be a lift of $S$, and $F_S$ a fundamental domain for the action of $\pi_1(S)$ on $\widetilde S$. Consider $\mathcal{F}_{\mathbb{H}^d}(F_S)$ of frames whose first $d-1$ vectors are tangent to $F_S$. Similarly, given any closed geodesic $\gamma$, let $\tilde\gamma$ be a lift of $\gamma$, and $F_\gamma$ a fundamental domain on $\tilde{\gamma}$ for the action of the corresponding loxodromic element stabilizing $\tilde\gamma$. We will denote the loxodromic element by $\tilde\gamma$ as well. Set $\mathcal{F}_{\mathbb{H}^d}(F_\gamma)$ to be the set of frames whose first vector is tangent to $F_\gamma$. We have
    \begin{align}
    \nonumber
    &c(\gamma_T,S)\int_{\mathcal{F}(M)} f(v) \int_{\mathcal{F}(M)} K_{\delta}(u,v) \, d\mu^1_{\gamma_T}(u)d\mu^1_{S}(v)\\
    \nonumber
    =&\frac{1}{|I(\gamma_T,S)|} \sum_{\ell(\gamma)\le T}\int_{ \mathcal{F}(S)} f(v) \int_{\mathcal{F}(M)}K_{\delta}(u,v)d\mu_{\gamma}(u)d\mu_S(v)\\
    \nonumber
    =&\frac{1}{|I(\gamma_T,S)|} \sum_{\ell(\gamma)\le T}\int_{\mathcal{F}_{\mathbb{H}^d}(F_S)} f(\tilde v)\sum_{g\in\Pi}\int_{\mathcal{F}_{\mathbb{H}^d}(F_\gamma)}\mathbb{K}_{\delta}(g.\tilde u,\tilde v)d\mu_{\tilde\gamma}(\tilde u)d\mu_{\widetilde{S}}(\tilde v)\\
    \label{equ: left_cosets}
    =&\frac{1}{|I(\gamma_T,S)|} \sum_{\ell(\gamma)\le T}\int_{\mathcal{F}_{\mathbb{H}^d}(F_S)} f(\tilde v)\sum_{g\in\Pi/\langle\tilde\gamma\rangle}\int_{\mathcal{F}_{\mathbb{H}^d}(\tilde\gamma)}\mathbb{K}_{\delta}(g.\tilde u,\tilde v)d\mu_{\tilde\gamma}(\tilde u)d\mu_{\widetilde{S}}(\tilde v)\\
    \label{equ: right_cosets}
    =&\frac{1}{|I(\gamma_T,S)|} \sum_{\ell(\gamma)\le T}\sum_{g\in\Pi/\langle\tilde\gamma\rangle}\int_{\mathcal{F}_{\mathbb{H}^d}(F_S)} f(\tilde v)\int_{\mathcal{F}_{\mathbb{H}^d}(\tilde\gamma)}\mathbb{K}_{\delta}(g.\tilde u,\tilde v)d\mu_{\tilde\gamma}(\tilde u)d\mu_{\widetilde{S}}(\tilde v)\\
    \label{equ: claim2.proof1}
    =&\frac{1}{|I(\gamma_T,S)|} \sum_{\ell(\gamma)\le T}\sum_{g\in\pi_1(S)\backslash\Pi/\langle\tilde\gamma\rangle}\int_{\mathcal{F}_{\mathbb{H}^d}(\tilde S)} f(\tilde v)\int_{\mathcal{F}_{\mathbb{H}^d}(\tilde\gamma)}\mathbb{K}_{\delta}(g.\tilde u,\tilde v)d\mu_{\tilde\gamma}(\tilde u)d\mu_{\widetilde{S}}(\tilde v)
    \end{align}
    We briefly explain Equality~\ref{equ: left_cosets} (the third equality above). Equality~\ref{equ: right_cosets} holds by similar arguments. First note that the expression
    $$\int_{\mathcal{F}_{\mathbb{H}^d}(\tilde\gamma)}\mathbb{K}_{\delta}(g.\tilde u,\tilde v)d\mu_{\tilde\gamma}(\tilde u)$$
    does not depend on the representative $g$ for a coset in $\Pi/\langle\tilde\gamma\rangle$. Indeed, this follows from the fact that the set $\mathcal{F}_{\mathbb{H}^d}(\tilde\gamma)$ and the measure $\mu_{\tilde\gamma}$ are both $\langle\tilde\gamma\rangle$-invariant. Because of the same invariance of the measure, we then have
    $$\sum_{k\in\mathbb{Z}}\int_{\mathcal{F}_{\mathbb{H}^d}(F_\gamma)}\mathbb{K}_{\delta}(g\tilde\gamma^k.\tilde u,\tilde v)d\mu_{\tilde\gamma}(\tilde u)=\sum_{k\in\mathbb{Z}}\int_{\mathcal{F}_{\mathbb{H}^d}(\tilde\gamma^k\cdot F_\gamma)}\mathbb{K}_{\delta}(g.\tilde u,\tilde v)d\mu_{\tilde\gamma}(\tilde u)=\int_{\mathcal{F}_{\mathbb{H}^d}(\tilde\gamma)}\mathbb{K}_{\delta}(g.\tilde u,\tilde v)d\mu_{\tilde\gamma}(\tilde u)$$
    for any $g\in\Pi$. Equality~\ref{equ: left_cosets} then follows by grouping the terms into $\langle\tilde\gamma\rangle$-cosets.
    
    Note that for any representative $g$ of an element in $\pi_1(S)\backslash\Pi/\langle\tilde\gamma\rangle$, the integral
    $$\int_{\mathcal{F}_{\mathbb{H}^d}(\tilde S)} f(\tilde v)\int_{\mathcal{F}_{\mathbb{H}^d}(\tilde\gamma)}\mathbb{K}_{\delta}(g.\tilde u,\tilde v)d\mu_{\tilde\gamma}(\tilde u)d\mu_{\widetilde{S}}(\tilde v)=\int_{\mathcal{F}_{\mathbb{H}^d}(\tilde S)} f(\tilde v)\int_{\mathcal{F}_{\mathbb{H}^d}(g.\tilde\gamma)}\mathbb{K}_{\delta}(\tilde u,\tilde v)d\mu_{\tilde\gamma}(\tilde u)d\mu_{\widetilde{S}}(\tilde v)$$
    is only nonzero when $\tilde S$ and $g.\tilde\gamma$ intersects transversely at $\tilde P$, which is a lift of an intersection point $P\in\mathscr{P}(\gamma,S)$. This in fact gives a one-to-one correspondence between elements in $g\in\pi_1(S)\backslash\Pi/\langle\tilde\gamma\rangle$ for which the integral is nonzero and points $P\in\mathscr{P}(\gamma,S)$.
    Therefore, Equation \ref{equ: claim2.proof1} is equal to
    \begin{align*}
    &\sum_{P\in \mathscr{P}(\gamma_T,S)} \frac{2\delta}{c_{\delta}|I(\gamma_T,S)|} \int_{\pi^{-1}( B^{\widetilde{S}}_{\delta}(\tilde P))} f(\tilde v)  d\mu_{\widetilde S}(\tilde v)\\
    =&\sum_{P\in \mathscr{P}(\gamma_T,S)} \frac{1}{\vol(B^{\widetilde{S}}_{\delta}(\tilde P))|I(\gamma_T,S)|} \int_{ \pi^{-1}(B^{\widetilde{S}}_{\delta}(\tilde P))}f(\tilde v)d\mu_{\widetilde S}(\tilde v).
    \end{align*}
    Now we have
    \begin{align}
    \nonumber
    &\left|\int_{M} f(x) \, dI(\gamma_T,S)/|I(\gamma_T,S)|-c(S,\gamma_T)\int_{\mathcal{F}(M)} f(v) \int_{\mathcal{F}(M)} K_{\delta}(u,v) \, d\mu^1_{\gamma_T}(u)d\mu^1_{S}(v)\right|\\
    \nonumber
    \leq &\sum_{P\in \mathscr{P}(\gamma_T,S)}\left|\frac{f(\tilde P)}{|I(\gamma_T,S)|}-\frac{1}{\vol(B^{\widetilde{S}}_{\delta}(\tilde P))|I(\gamma_T,S)|} \int_{\pi(\tilde v)\in B^{\widetilde{S}}_{\delta}(\tilde P)}f(\tilde v)d\mu_{\widetilde S}(\tilde v)\right|
    \\
    \nonumber
    \leq &\frac{1}{|I(\gamma_T,S)|}  \sum_{P\in \mathscr{P}(\gamma_T,S)} \frac{1}{\vol(B^{\widetilde{S}}_{\delta}(\tilde P))} \int_{\pi(\tilde v)\in B^{\widetilde{S}}_{\delta}(\tilde P)}|f(\tilde P)-f(\tilde v)|d\mu_{\widetilde S}(\tilde v)\\
    \nonumber
    \leq &\frac{1}{|I(\gamma_T,S)|}  \sum_{P\in \mathscr{P}(\gamma_T,S)} \frac{\Lip(f)}{\vol( B^{\widetilde{S}}_{\delta}(\tilde P))} \int_{\pi(\tilde v)\in B^{\widetilde{S}}_{\delta}(\tilde P)} \delta \, d\mu_{\widetilde S}(\tilde v)\leq \delta \, \Lip(f),
    \end{align} 
    as desired.
\end{proof}
\begin{constr}\label{constr: inter_kernel}
    $\delta$ should go to zero as $T,\vol(S)\to\infty$.
\end{constr}

\paragraph{Step 2: Smooth the kernel}
We wish to apply the equidistribution results in \S\ref{sec:equidistri}. However, the intersection kernel is not smooth; in fact, it is clearly not even continuous, as it is only a positive constant within a $\delta$-neighborhoods of the base points. While it is possible to smooth $K_\delta$ by taking convolution with a standard mollifier, it is better for later arguments to smooth it in a way so that the contribution of $u$ and $v$ are still separate.

One could try to smooth the kernel by taking the product of smoothed characteristic functions of the neighborhoods appearing in its definition. However, the result would still be discontinuous, as explained below.

Fix a frame $v$, and suppose $\delta>0$ is small enough. Consider a frame $u$ so that the base point of $\phi_\delta(u)$ lies on the boundary of $B_\delta^{d-1}(v)$. By definition of the intersection kernel, we know that $K_\delta(u,v)\neq0$. Choose a sequence of such frames $\{u_n\}$ so that the angle of intersection between $B_\delta^1(u_n)$ and $\overline{B_\delta^{d-1}(v)}$ goes to zero. Up to a subsequence, we may assume $u_n\to u_0$ for some frame $u_0$ lying on the hyperplane tangent to $v$. Then we have $K_\delta(u_n,v)=\text{constant}\neq0$, while $K_\delta(u_0,v)=0$ (since $B_{\delta}^1(u)$ doesn't intersect $B_{\delta}^{d-1}(v)$ transversely).

Moreover, one can arrange so that $d(\pi(v),\pi(u_0))=2\delta$, so it is easy to construct a different sequence $u_n'\to u_0$ so that $K_\delta(u_n',v)=0$ (by considering $u_n'$ almost parallel to the hyperplane tangent to $v$). Note that this issue arises when $\pi(u)$ is close to the geodesic hyperplane tangent to $v$. We resolve this by taking away a small neighborhood of $B_{2\delta}^{d-1}(v)$ as follows.

Fix $o\in\mathbb{H}^d$, and let $v_0$ be a frame based at $o$. Let $\rho\in (0,\delta)$ be a positive parameter to be determined later. Let $\mathcal{N}_{\delta,\rho}(v_0)$ be the $\rho$-neighborhood of $B_{2\delta}^{d-1}(v_0)$. Note that $\mathcal{N}_{\delta,\rho}(v_0)$ is convex and rotationally symmetric around the axis passing through $\pi(v_0)$ tangent to the last vector of $v_0$ (see Figure \ref{fig: N}).
\begin{figure}[htp]
    \centering

\tikzset{every picture/.style={line width=0.75pt}} 

\begin{tikzpicture}[x=0.75pt,y=0.75pt,yscale=-0.8,xscale=0.8]

\draw   (100,132.2) .. controls (119.22,118.95) and (152.75,108.2) .. (174.9,108.2) .. controls (197.05,108.2) and (199.42,118.95) .. (180.2,132.2) .. controls (160.98,145.45) and (127.45,156.2) .. (105.3,156.2) .. controls (83.15,156.2) and (80.78,145.45) .. (100,132.2) -- cycle ;
\draw  [line width=3] [line join = round][line cap = round] (138.2,128.2) .. controls (138.2,128.2) and (138.2,128.2) .. (138.2,128.2) ;
\draw  [dash pattern={on 4.5pt off 4.5pt}] (60.17,130.2) .. controls (78.38,102.59) and (128.97,80.2) .. (173.17,80.2) .. controls (217.37,80.2) and (238.44,102.59) .. (220.23,130.2) .. controls (202.02,157.81) and (151.43,180.2) .. (107.23,180.2) .. controls (63.03,180.2) and (41.96,157.81) .. (60.17,130.2) -- cycle ;
\draw  [dash pattern={on 4.5pt off 4.5pt}] (40.99,131.57) .. controls (33.68,94.36) and (73.07,64.2) .. (128.98,64.2) .. controls (184.88,64.2) and (236.12,94.36) .. (243.43,131.57) .. controls (250.74,168.79) and (211.34,198.95) .. (155.44,198.95) .. controls (99.53,198.95) and (48.29,168.79) .. (40.99,131.57) -- cycle ;
\draw  [dash pattern={on 4.5pt off 4.5pt}] (95,131.2) .. controls (73.02,96.41) and (75.66,68.2) .. (100.9,68.2) .. controls (126.15,68.2) and (164.42,96.41) .. (186.4,131.2) .. controls (208.38,165.99) and (205.74,194.2) .. (180.5,194.2) .. controls (155.25,194.2) and (116.97,165.99) .. (95,131.2) -- cycle ;
\draw    (137.05,128.04) -- (191.79,114.25) ;
\draw    (89.45,142.56) -- (105.3,156.2) ;
\draw    (93.45,137.56) -- (113.45,155.56) ;
\draw    (97.34,132.2) -- (122.45,154.56) ;
\draw    (105.45,129.56) -- (132.45,152.56) ;
\draw    (112.45,125.56) -- (141.45,150.56) ;
\draw    (120.45,120.56) -- (151.45,147.56) ;
\draw    (130.45,117.56) -- (159.45,143.56) ;
\draw    (139.71,113.75) -- (168.45,139.56) ;
\draw    (148.45,111.56) -- (174.45,136.56) ;
\draw    (157.45,110.56) -- (180.2,132.2) ;
\draw    (168.93,107.53) -- (186.71,126.75) ;
\draw    (179.93,108.53) -- (191.68,120.8) ;
\draw    (91.2,70.2) -- (137.05,128.04) ;
\draw  [color={rgb, 255:red, 255; green, 255; blue, 255 }  ,draw opacity=1 ][fill={rgb, 255:red, 255; green, 255; blue, 255 }  ,fill opacity=1 ] (155,121.38) -- (167.48,121.38) -- (167.48,131.53) -- (155,131.53) -- cycle ;
\draw  [color={rgb, 255:red, 255; green, 255; blue, 255 }  ,draw opacity=1 ][fill={rgb, 255:red, 255; green, 255; blue, 255 }  ,fill opacity=1 ] (126.95,129.36) -- (136.15,129.36) -- (136.15,138.06) -- (126.95,138.06) -- cycle ;

\draw (122.45,123.96) node [anchor=north west][inner sep=0.75pt]  [font=\footnotesize]  {$v_{0}$};
\draw (157.08,119.12) node [anchor=north west][inner sep=0.75pt]  [font=\footnotesize]  {$\delta $};
\draw (111,79.4) node [anchor=north west][inner sep=0.75pt]  [font=\footnotesize]  {$\rho $};

\end{tikzpicture}

    \caption{$\mathcal{N}_{\delta,\rho}(v_0)$ is the $\rho-$neighborhood of $B_{2\delta}^{d-1}(v_0)$.}
    \label{fig: N}
\end{figure}
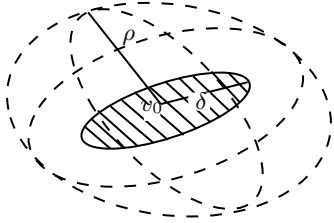
Now for any $v\in\mathcal{F}(\mathbb{H^d})$, there exists an isometry $g\in\isom^+(\mathbb{H}^d)$ sending $v_0$ to $v$. Define $\mathcal{N}_{\delta,\rho}(v)$ as the image of $\mathcal{N}_{\delta,\rho}(v_0)$ under this isometry.

Suppose $t \in (0,\rho)$ is another parameter to be determined later. Let $\chi_{\delta}^{t,\rho}$ be a smooth function identically $1$ outside $\mathcal{N}_{\delta,\rho}(v_0)$ and identically $0$ on $\mathcal{N}_{\delta,\rho-t}(v_0)$. Moreover, we can guarantee that $m$-th derivatives of $\chi_{\delta}^{t,\rho}$ are of the order $O(t^{-m})$. This can be done by taking the convolution of the characteristic function of the complement of $\mathcal{N}_{\delta,\rho-t/2}(v_0)$ with a smooth mollifier supported in a ball of radius $t/2$. We can then push forward the function $\chi_{\delta}^{t,\rho}$ to a function identically $1$ outside $\mathcal{N}_{\delta,\rho}(v)$ via the isometry $g$. We denote this function by $\chi_{v,\delta}^{t,\rho}$.

Let $\eta_\delta^{t}:\mathbb{R}\to\mathbb{R}$ be a smooth even function supported on the interval $[-\delta,\delta]$ and identically $1$ in $[-\delta+t,\delta-t]$. We also assume that $\eta_\delta^{t}$ is monotone on $\mathbb{R}_+$. Again, the $m$-th derivatives of $\eta_\delta^{t}$ can be chosen to have order $O(t^{-m})$. Define 
\begin{align*}
\mathbb{K}_{\delta}^{t,\rho}(u,v)&=
\begin{cases}
    \frac{1}{c_\delta}\chi_{v,\delta}^{t,\rho}(\pi(u))\cdot\eta_\delta^{t}(d(P,\pi(v)))&\text{if $P=\pi(\phi_{\tau_0}(u))\in B^{d-1}_{\delta}(v)$ for some $|\tau_0|<\delta$},\\
    0 &\text{otherwise}. 
\end{cases}
\\
\widetilde{\mathbb{K}}_{\delta}^{t,\rho}(u,v)&=
\begin{cases}
    \mathbb{K}_{\delta}^{t,\rho}(u,v)\cdot\eta^{t}_\delta(\tau_0)&\text{if $P=\pi(\phi_{\tau_0}(u))\in B^{d-1}_{\delta}(v)$ for some $|\tau_0|<\delta$},\\
    0 &\text{otherwise.}
\end{cases}
\end{align*}
Roughly speaking, $\mathbb{K}_\delta^{t,\rho}$ becomes zero whenever the base point of $u$ is close to $B_{2\delta}^{d-1}(v)$, and is smooth in the $v$ variable; $\widetilde{\mathbb K}_{\delta}^{t,\rho}(u,v)$ is further smooth in the $u$ variable.

Clearly, both are bounded, invariant under isometries, and have compact support when the second input is fixed. So we can define the corresponding kernels $K_\delta^{t,\rho}$ and $\widetilde{K}_\delta^{t,\rho}$ on $\mathcal{F}(M)\times\mathcal{F}(M)$ as in Equation \ref{def: K} in \S\ref{sec: inter_kernel}.

We collect some properties of these kernels below.
\begin{Lemma}\label{lem:kernel_smooth_prop}
\begin{enumerate}[label=\normalfont{(\arabic*)}]
    \item For a fixed $v$, $\mathbb{K}_\delta^{t,\rho}(u,v)$ differs from $\mathbb{K}_\delta(u,v)$ over a region of volume $\asymp\delta^{d-1}\rho$ if $t=O(\rho^2)$ and $\delta=O(1)$.
    \item $\widetilde{\mathbb{K}}_\delta^{t,\rho}$ is smooth. Consequently,
     $\widetilde K_\delta^{t,\rho}$ is smooth.
    \item The $m$-th derivatives of $\widetilde{\mathbb{K}}_\delta^{t,\rho}$ have size $O(\delta^{-d}t^{-m})$. Consequently, over $M_{\ge\iota}$, the $m$-th derivatives of $\widetilde{K}_\delta^{t,\rho}$ have size $O(\lceil\delta/\iota\rceil^{d-1}\delta^{-d}t^{-m})$.
    \item For a fixed $v$, $\widetilde{\mathbb{K}}_\delta^{t,\rho}(u,v)$ is nonzero on a region with volume $\asymp \delta^d$ if $\rho/\delta \to 0$ as $\delta \to 0$, which is true in this paper based on how we pick $\rho$ later. 
\end{enumerate}
\end{Lemma}
\begin{proof}
  For Part (1), note that they differ either when $\pi(u)\in\mathcal{N}_{\delta,\rho}(v)$, or when $B_\delta^1(u)$ meets the region $B_\delta^{d-1}(v)-B_{\delta-t}^{d-1}(v)$. The former region clearly has volume proportional to the volume of $\mathcal{N}_{\delta,\rho}(v)$, which has size $\asymp \delta^{d-1}\rho$. The latter is contained in the following region. The base point is clearly contained in the geodesic flow box given by flowing $B_{2\delta}^{d-1}(v)$ along its normal vector forward and backward by distance $\delta$. For each base point, the portion of those vectors meeting $B_\delta^{d-1}(v)-B_{\delta-t}^{d-1}(v)$ via forward or backward geodesic flow is at most $O(t/\rho)$ (since $\pi(u)$ is not in $(\rho-t)$-neighborhood of $B_\delta^{d-1}(v)$). Hence the volume of the latter region is $O(\delta^d t/\rho)=O(\delta^{d}\rho)=O(\delta^{d-1}\rho)$. This gives Part (1).

    For Part (2), $\widetilde{\mathbb{K}}_\delta^{t,\rho}(u,v)$ is invariant under any isometry $g$ of $\mathbb{H}^d$, i.e.~$\widetilde{\mathbb K}_\delta^{t,\rho}(g_*u,g_*v)=\widetilde{\mathbb K}_\delta^{t,\rho}(u,v)$. Hence it suffices to show $\widetilde{K}_\delta^{t,\rho}(u,v)$ is smooth in $u$ after fixing $v$. This is clear from the definition.

    The estimates for the derivatives of $\widetilde{\mathbb{K}}_\delta^{t,\rho}$ in Part (3) follow from the choice of $\chi_{v,\delta}^{t,\rho}$ and $\eta_\delta^t$, as well as the fact that $c_\delta\asymp\delta^d$. Combined with Proposition~\ref{prop: growth}, we then have the estimates for $\widetilde K_\delta^{t,\rho}$.

    For Part (4), note that this region is the same as the region where $\mathbb{K}_{\delta}^{t, \rho}(u,v)$ is nonzero, and it suffices to estimate the region for $\mathbb{K}_\delta$, as by Part (1) the difference would be $o(\delta^{d-1}\rho)=o(\delta^{d})$. The set of $u$'s in the support can be characterized by $(P,\theta,\tau_0)$ where $P,\tau_0$ are as in the definition of $\mathbb{K}_{\delta}^{t,\rho}$, and $\theta$ presents the angle between $B_{\delta}^{d-1}(v)$ and $B_{\delta}^1(u)$. The set of possible $P$'s is $B_{\delta}^{d-1}(v)$, angle $\theta$ can be anything in $[0,\pi]$ and $\tau_0$ ranges from $-\delta$ to $\delta$, so the volume is $\vol_{d-1}(\delta)\cdot \pi \cdot 2\delta  \asymp \delta^d$.
\end{proof}
We are now ready to replace the intersection kernel by the smoothed one.
\begin{Claim}\label{claim: smoothing}
Suppose $t=O(\rho^2)$ and $\delta=O(1)$. Choose $\iota$ small enough so that $\supp(f)\subseteq M_{\ge\iota}$. We have
    \begin{align}\label{equ: statement-smoothing}
        &\left|\int_{\mathcal{F}(M)} f(v) \int_{\mathcal{F}(M)} \widetilde{K}^{t,\rho}_{\delta}(u,v) \, d\mu^1_{\gamma_T}(u)d\mu^1_{S}(v)-\int_{\mathcal{F}(M)} f(v) \int_{\mathcal{F}(M)} K_{\delta}(u,v) \, d\mu^1_{\gamma_T}(u)d\mu^1_S(v)\right| \nonumber \\ 
        \leq &  C\lceil\delta/\iota\rceil^{d-1}\sup|f|\left(\frac{t}{\delta^{d+1}}+\frac{\rho^q}{t^q\delta^d}e^{-\epsilon_1 T}\right)+C\sup|f|\frac{\rho}{\delta\vol(M)}. 
        \end{align}
where $q,\epsilon_1$ are the constants from Theorem~\ref{thm: mixing.geod.3}.
\end{Claim}
\begin{proof}

From the construction of the smoothing, we conclude that the left hand side of Equation \ref{equ: statement-smoothing} satisfies
\begin{align}\label{line: K.delta.tt'}
    \nonumber
    \le&\int_{\mathcal{F}(M)} |f(v)| \int_{\mathcal{F}(M)} |\widetilde{K}_{\delta}^{t,\rho}(u,v)-K^{t,\rho}_{\delta}(u,v)| d\mu_{\gamma_T}^1(u)d\mu_S^1(v)\\ &\qquad\qquad\qquad+\int_{\mathcal{F}(M)} |f(v)| \int_{\mathcal{F}(M)}|K_{\delta}^{t,\rho}(u,v)-K_{\delta}(u,v)|d\mu_{\gamma_T}^1(u)d\mu_S^1(v)
\end{align}

For the first inner integral, note that by definition, on a geodesic segment of length $2\delta$, the length of the segment where $\mathbb{K}_\delta^{t,\rho}$ and $\widetilde{\mathbb K}_\delta^{t,\rho}$ differ is at most $2t$. Similar to Proposition~\ref{prop: growth}, the number of overlaps when we move from the universal cover to the manifold itself is $O((\delta/\iota)^{d-1})$. Hence the first inner integral is bounded above by $O(t\lceil\delta/\iota\rceil^{d-1}/(c_\delta\delta))$.

For the second inner integral, consider the region where $\mathbb{K}_\delta^{t,\rho}(u,v)$ differs from $\mathbb{K}_\delta(u,v)$ for a fixed $v$. Based on Part (1) of the previous lemma, this region has volume $\asymp\delta^{d-1}\rho$. Moreover, based on the description of the region in the proof of the lemma, we can find a smooth function $\psi_v(u)$ identically $1$ over this region, and has $m$-th derivatives of size $O((t/\rho)^{-m})$. Set $\mathbf{\Psi}(u,v):=\psi_v(u)/c_\delta$, and let $\psi(u,v)$ to be the corresponding kernel on $\mathcal{F}(M)$. Then $|\mathbb{K}_\delta^{t,\rho}(u,v)-\mathbb{K}_\delta(u,v)|\le \mathbf{\Psi}(u,v)$, and hence $|K_\delta^{t,\rho}(u,v)-K_\delta(u,v)|\le\psi(u,v)$.

Now by equidistribution of the geodesic flow (Theorem \ref{thm: mixing.geod.3}) we conclude that the second inner integral is bounded above by  
$$
\int_{\mathcal{F}(M)}\psi(u,v)dL^1_M(u)+O(\lceil\delta/\iota\rceil^{d-1}(t/\rho)^{-q}e^{-\epsilon_1 T})\le \frac{C\rho\delta^{d-1}}{c_\delta\vol(M)}+O(\lceil\delta/\iota\rceil^{d-1}(t/\rho)^{-q}e^{-\epsilon_1 T}c_\delta^{-1}).
$$
where we make use of Part (3) of Proposition~\ref{prop: growth}.

Therefore, the sum in $(\ref{line: K.delta.tt'})$ is
$$
\leq C\lceil\delta/\iota\rceil^{d-1}\sup|f|\left(\frac{t}{\delta^{d+1}}+\frac{\rho^q}{t^q\delta^d}e^{-\epsilon_1 T}\right)+C\sup|f|\frac{\rho}{\delta\vol(M)},
$$
as desired.
\end{proof}
\begin{constr}\label{constr: smoothing}
    Generally, we will ensure $\delta<\iota$, and in that case, we need $t=O(\rho^2)$, $\delta=O(1)$, $t/\delta^{d+1}\to0$, $\rho/\delta\to0$, and $\rho^q/(t^q\delta^d)e^{-\epsilon_1T}\to0$.
\end{constr}

\paragraph{Step 3: Apply the equidistribution results}
We can now apply the two equidistribution results to get integration over the entire frame bundle with respect to the uniform measures.
\begin{Claim}\label{claim: equi_geod}
Suppose $\supp(f)\subseteq M_{\ge\iota}$. Replacing $\mu_{\gamma_T}^1$ with $L_M^1$, we have
\begin{align}
       &\left|\int_{\mathcal{F}(M)} f(v) \int_{\mathcal{F}(M)} \widetilde{K}^{t,\rho}_{\delta}(u,v) \, d\mu^1_{\gamma_T}(u)d\mu^1_{S}(v)-\int_{\mathcal{F}(M)} f(v) \int_{\mathcal{F}(M)} \widetilde{K}^{t,\rho}_{\delta}(u,v) \, dL^1_M(u)d\mu^1_{S}(v)\right| \nonumber\\
       \leq& C\lceil\delta/\iota\rceil^{d-1}\cdot \sup|f| \cdot \delta^{-d}t^{-q}e^{-\epsilon_1 T}.
       \end{align}
\end{Claim}
\begin{proof}
    It is a direct consequence of the equidistribution of the geodesic flow, Theorem \ref{thm: mixing.geod.3}; the difference is
    $$
    \leq \int_{\mathcal{F}(M)} |f(v)| O(\Sob_q(\widetilde{K}_{\delta}^{t,\rho})e^{-\epsilon_1 T})d\mu_S^1(v)\leq \sup|f|\cdot O(\Sob_q(\widetilde{K}_{\delta}^{t,\rho})e^{-\epsilon_1 T}).
    $$
    The final estimate then follows from Part (3) of Lemma~\ref{lem:kernel_smooth_prop}.    
\end{proof}
\begin{constr}\label{constr: equi_geod}
    We need $\delta^{-d}t^{-q}e^{-\epsilon_1 T}\to0$.
\end{constr}
\begin{Claim}\label{claim: equi_hyper}
Replacing $\mu_S^1$ with $L_M^1$, we have 
    \begin{align}
\nonumber
&\left|\int_{\mathcal{F}(M)} f(v) \int_{\mathcal{F}(M)} \widetilde{K}^{t,\rho}_{\delta}(u,v) \, dL^1_M(u)d\mu^1_{S}(v)-\int_{\mathcal{F}(M)} f(v) \int_{\mathcal{F}(M)} \widetilde{K}^{t,\rho}_{\delta}(u,v) \, dL^1_M(u)dL^1_M(v)\right| \\ 
\leq& C \vol(S)^{-\epsilon_2} \Sob_q(f).
\label{eq: claim 5.4}
\end{align}
where $q,\epsilon_2>0$ are the constants from Theorem~\ref{thm: mixing.surface}.
\end{Claim}
\begin{proof}
By Theorem \ref{thm: mixing.surface}, the left hand side of Equation \ref{eq: claim 5.4} is 
\begin{align*}
\le C\vol(S)^{-\epsilon_2}\Sob_q\left(f(v) \int_{\mathcal{F}(M)} \widetilde{K}^{t,\rho}_{\delta}(u,v) \, dL^1_M(u)\right).
\end{align*}
Note that
$$
\int_{\mathcal{F}(M)} \widetilde{K}^{t,\rho}_{\delta}(u,v) \, dL^1_M(u)=\frac1
{\vol(M)}\int_{\mathcal{F}(\mathbb{H}^d)}\widetilde{\mathbb K}^{t,\rho}_{\delta}(u,v) \, dL_{\mathbb{H}^d}(u)
$$
is a constant function of $v$ by Proposition~\ref{prop: growth}. Moreover, for a fixed $v$, $\widetilde{\mathbb K}^{t,\rho}_{\delta}$ is nonzero on a region with volume $\asymp \delta^d$ by Part (4) of Lemma \ref{lem:kernel_smooth_prop}, and its supremum is $\asymp1/\delta^d$, so this integral is in fact $\asymp 1$.
\end{proof}
\begin{constr}\label{constr: equi_hyper}
    We need $\vol(S)\to\infty$, which is part of the assumption.
\end{constr}

\paragraph{Step 4: Bound the coefficient}
To obtain an upper bound on the final error terms, we need to bound $c(\gamma_T,S)$ from above, independent of $T$ and $\vol(S)$. Let $h_\iota$ be a smooth function equal to $1$ on $M_{\ge\iota}$ and $0$ over $M_{\iota/2}$. Moreover, note that the hyperbolic distance between $M_{\iota/2}$ and $M_{\geq \iota}$ is $\asymp \log(2)$, therefore, we can arrange so that $0\le h_\iota\le 1$ and all derivatives of $h_\iota$ are bounded above independent of $\iota$. 

\begin{Claim}\label{claim: coeff}
Let $A_2,A_3$ be the upper bounds we obtained in Claims~\ref{claim: smoothing} and \ref{claim: equi_geod} for the function $h_\iota$. Then we have
$$
\left|1-c(\gamma_T,S)\int_{\mathcal{F}(M)}\widetilde{K}^{t,\rho}_{\delta}(u,v)dL_M^1(u)\right|\leq C\delta+c(\gamma_T,S)(A_2+A_3+O(\iota^{d-1}))+C\vol(S)^{-\epsilon_2}\iota^{-q}.
$$
Consequently, if $A_2,A_3=o(1)$, $\delta,\iota\to0$, and $\vol(S)^{-\epsilon_2}\iota^{-q}\to0$, then $c(\gamma_T,S)=O(1)$ as $T\to\infty$. 
\end{Claim}
\begin{proof}
For better presentation, we define

        \begin{align*}
        &E_1:=\int_{\mathcal{F}(M)}h_\iota(v)\int_{\mathcal{F}(M)}K_{\delta}(u,v)\, d\mu_{\gamma_T}^1(u)d\mu_S^1(v) \\
        &E_2:=\int_{\mathcal{F}(M)}h_\iota(v)\int_{\mathcal{F}(M)}\widetilde{K}^{t,\rho}_{\delta}(u,v)\, d\mu_{\gamma_T}^1(u)d\mu_S^1(v) \\
        &E_3:= \int_{\mathcal{F}(M)}h_\iota(v)\int_{\mathcal{F}(M)}\widetilde{K}^{t,\rho}_{\delta}(u,v)\, dL_M^1(u)d\mu_S^1(v) \\
        &E_4:=\int_{\mathcal{F}(M)}h_\iota(v)\int_{\mathcal{F}(M)}\widetilde{K}^{t,\rho}_{\delta}(u,v)\, dL_M^1(u)dL_M^1(v)\\
        \end{align*}

        Consider the inequalities in Claims \ref{claim: inter_kernel}--\ref{claim: equi_hyper} for $f=h_{\iota}$, then we have
        \begin{align}
        &\left|1-c(S,\gamma_T)\int_{\mathcal{F}(M)}\widetilde{K}^{t,\rho}_{\delta}(u,v)dL_M^1(u) \right|\nonumber\\
        \leq &\left|1-\int_Mh_\iota\,dI^1(\gamma_T,S)\right|\nonumber+\left|\int_Mh_\iota\,dI^1(\gamma_T,S)-c(\gamma_T,S)E_1\right|\nonumber\\
        &+c(\gamma_T,S)\left(\left|E_1-E_2 \right|+\left|E_2-E_3\right|+\left|E_3-E_4 \right| + \left|E_4-\int_{\mathcal{F}(M)}\widetilde{K}^{t,\rho}_{\delta}(u,v)dL_M^1(u)\right| \right)  \nonumber \\
        \leq& \frac{|I_\iota(\gamma_T,S)|}{|I(\gamma_T,S)|}+\delta\Lip(h_\iota)+c(\gamma_T,S)(A_2+A_3)+C\vol(S)^{-\epsilon_2}\Sob_q(h_\iota) \\  & +c(\gamma_T,S)\int_{\mathcal{F}(M)}\widetilde{K}^{t,\rho}_{\delta}(u,v)dL_M^1(u) \cdot \frac{\vol(M_{\iota})}{\vol(M)}
        \end{align}
        by applying the previous claims. Here we used the fact that the integral $\int_{\mathcal{F}(M)}\widetilde{K}^{t,\rho}_{\delta}(u,v)\, dL_M^1(u)$ is constant and does not depend on $v$. Moreover, this integral is bounded below by a positive constant, since it is the integral of a function that is $\asymp \delta^{-d}$ in a region with volume $\asymp \delta^{d}$. This is needed to show $c(\gamma_T,S)=O(1)$.
        
        Now by Theorem~\ref{thm: cusp_inter_points}, we have 
        $$\frac{|I_\iota(\gamma_T,S)|}{|I(\gamma_T,S)|}=c(\gamma_T,S)\frac{|I_\iota(\gamma_T,S)|}{\ell(\gamma_T)\vol(S)}\le c(\gamma_T,S)\cdot C\iota^{d-2}.$$
        Moreover, $\Lip(h_\iota)$ is chosen to be uniformly bounded independent of $\iota$, and $\Sob_q(h_\iota)$ increases in the cusp by a weight $\asymp (1/\iota)^q$ (see the relevant discussions in \cite{effective.surface}). Combining these, we have the desired estimate.
    \end{proof}

\begin{constr}\label{constr: coeff}
    Need $\delta, A_2,A_3=o(1)$, which are already covered by previous constraints. Note that we generally arrange $\delta<\iota$. We also need $\iota\to 0$ and $\vol(S)^{-\epsilon_2}\iota^{-q}\to0$.
\end{constr}

\paragraph{Step 5: Proof of the main theorem for $k=1$}
\begin{Lemma}\label{lem:constr}
    Suppose $0<\kappa<\frac{\epsilon_1}{d+qd+2q}$ and choose any $q'>q$. Set $\delta=e^{-\kappa T}$, $\rho=e^{-(1+d/2)\kappa T}$, $t=\rho^2=e^{-(d+2)\kappa T}$, and $\iota=\max\{2\delta,\vol(S_n)^{-\epsilon_2/q'}\}$, then all the constraints are satisfied.
\end{Lemma}
\begin{proof}
    Clearly, Constraint~\ref{constr: inter_kernel} is satisfied. For Constraint~\ref{constr: smoothing}, we have $t=\rho^2$, $\delta\to0$, $t/\delta^{d+1}=e^{-\kappa T}\to0$, $\rho/\delta=e^{-\kappa d T/2}\to0$, and $\rho^q/(t^q\delta^d)e^{-\epsilon_1T}=e^{(qd/2+q+d)\kappa T}e^{-\epsilon_1T}\to 0$ since $\kappa<\frac{\epsilon_1}{qd/2+q+d}$. For Constraint~\ref{constr: equi_geod}, we have $\delta^{-d}t^{-q}e^{-\epsilon_1T}=e^{(d+qd+2q)\kappa T}e^{-\epsilon_1T}\to0$ since $\kappa<\frac{\epsilon_1}{d+qd+2q}$. Constraint~\ref{constr: equi_hyper} is automatically satisfied since we have a distinct sequence of hypersurfaces. Finally, for constraint \ref{constr: coeff}, $\iota>\delta$ by definition, and $\vol(S_n)^{-\epsilon_2}\iota^{-q}\le \vol(S_n)^{-\epsilon_2(1-q/q')}\to0$.
\end{proof}
\begin{proof}[Proof of Theorem \ref{thm: main} for $k=1$]
Consider a function $f \in C_c^{\infty}(M)$. Assume that $B_i$ is the upper bound in Claim $5.i$, for $i=1,\dots,4$ for $fh_{\iota}$, and $A$ is the upper bound found in Claim~\ref{claim: coeff}.
For better presentation, we define
\begin{align*}
    &E_1:=\int_{\mathcal{F}(M)}f(v)h_\iota(v) \int_{\mathcal{F}(M)}K_{\delta}(u,v)\, d\mu_{\gamma_T}^1(u)d\mu_{S_n}^1(v)\\
    &E_2:=\int_{\mathcal{F}(M)}f(v)h_\iota(v)\int_{\mathcal{F}(M)}K_{\delta}^{t,\rho}(u,v)\, d\mu_{\gamma_T}^1(u)d\mu_{S_n}^1(v)\\
    &E_3:=\int_{\mathcal{F}(M)}f(v)h_\iota(v)\int_{\mathcal{F}(M)}K_{\delta}^{t,\rho}(u,v)\, dL_M^1(u)d\mu_{S_n}^1(v)\\
    &E_4:= \int_{\mathcal{F}(M)}f(v)h_\iota(v)\int_{\mathcal{F}(M)}K_{\delta}^{t,\rho}(u,v)\, dL_M^1(u)dL_M^1(v)
\end{align*}
Now we have
\begin{align}\label{equ: upperbound}
      &\left|\int_M f \, dI^1(\gamma_T,S_n)-\int_M f\, d\vol^1\right|\nonumber\\
      \leq& \int_{M_\iota}|f|\, dI^1(\gamma_T,S_n)+\int_{M_\iota}|f|\, d\vol^1 +\left|\int_M f h_\iota\, dI^1(\gamma_T,S_n)-
       c(\gamma_T,S_n)E_1\right|\nonumber\\
       &+c(\gamma_T,S_n)\left(|E_1-E_2|+|E_2-E_3|+|E_3-E_4|\right)\nonumber+\left|c(\gamma_T,S_n)E_4-\int_{\mathcal{F}(M)} f(v)h_\iota(v) \, dL_M^1\right|  \nonumber \\
        \leq&C\sup_{M_{\iota}}|f|c(\gamma_T,S)\iota^{d-2}+C\|f\|_{L^2}\iota^{d-1}+B_1+c(\gamma_T,S_n)(B_2+B_3+B_4)+A\|f\|_{L^2}.
\end{align}
In the last inequality, for the first integral we used Theorem \ref{thm: cusp_inter_points}, and for the second integral we used $\vol(M_{\iota})\asymp \iota^{d-1}$. Moreover, using the constraints in Lemma~\ref{lem:constr} and noting that $\Sob_q(f)$ dominates $\sup|f|$ and $\Lip(f)$, we conclude that the Equation \ref{equ: upperbound} is bounded above by
$$\le C\Sob_q(f)(e^{-\epsilon_1T}+\vol(S)^{-\epsilon_2}),$$
for some positive constants $C, q,\epsilon_1,\epsilon_2$ depending only on the geometry of $M$.
\end{proof}

\section{Effective estimates II: the general case}\label{sec:estimates2}
Now we move to the general case. To avoid the case of closed geodesics already discussed, suppose $k\ge2$. 
Note that most of the proofs are similar to the case $k=1$ so we leave them to the reader, but the smoothing method and some of the bounds are more technical, which we will explain in detail.
We remark that the final result in the non-compact case is weaker than the result for $k=1$. 

Let $\Gamma$ be a maximal geodesic $k$-submanifold and $S$ a maximal geodesic $(d-k)$-submanifold. Define
$$c(\Gamma,S)=\frac{\vol(\Gamma)\vol(S)}{|I(\Gamma,S)|}.$$
Since we want to restrict to a fixed compact part (fixed $\iota$), we also define
$$c_{\iota}(\Gamma,S)=\frac{\vol(\Gamma)\vol(S\cap M_{\ge\iota})}{|I_{\ge\iota}(\Gamma,S)|}.$$
We also define a smooth version of $c_{\iota}$ as follows. Let $h_{\iota,\upsilon}$ be a smooth function equal to $1$ on $M_{\ge\iota}$ and $0$ on $M_{\iota-\upsilon}$ for $\iota>\upsilon>0$. We can arrange that the Lipschitz constant of $h_{\iota,\upsilon}$ is $(\log(\iota)-\log(\iota-\upsilon))^{-1}=O(\iota/\upsilon)$. Set $I_{\iota,\upsilon}(\Gamma,S):=h_{\iota,\upsilon}I(\Gamma,S)$, $\mu_{S,\iota,\upsilon}:=h_{\iota,\upsilon}\cdot\mu_S$. Then set
$$c_{\iota,\upsilon}(\Gamma,S)=\frac{\vol(\Gamma)|\mu_{S,\iota,\upsilon}|}{|I_{\iota,\upsilon}(\Gamma,S)|}.$$
Clearly, for a fixed pair of $\Gamma$ and $S$, $c_{\iota,\upsilon}\to c_\iota$ as $\upsilon\to 0$.

\paragraph{Step 1: Relate to the intersection kernel}
We have the following generalization of Claim~\ref{claim: inter_kernel}.
\begin{Claim}\label{claim: inter_kernel_gen}
    We have
    $$
    \left|\int_{M} f(x) \, dI^1(\Gamma,S)-c(\Gamma,S)\int_{\mathcal{F}(M)} f(v) \int_{\mathcal{F}(M)} K_{\delta}(u,v)\, d\mu^1_{\Gamma}(u)d\mu^1_{S}(v)\right|
    \leq \delta \, \Lip(f).
    $$
    For the smoothed version, we have
    \begin{align*}
    &\left|\int_{M} f(x) \, dI^1_{\iota,\upsilon}(\Gamma,S)-c_{\iota,\upsilon}(\Gamma,S)\int_{\mathcal{F}(M)} f(v) \int_{\mathcal{F}(M)} K_{\delta}(u,v)\, d\mu^1_{\Gamma}(u)d\mu^1_{S,\iota,\upsilon}(v)\right|
    \\&\leq \delta \, \Lip(f)+C(\iota)\cdot\delta\iota /\upsilon\cdot\sup_{\mathcal{R}}|f|,
    \end{align*}
    where $\mathcal{R}=M_{\iota e^{\delta}}\backslash M_{(\iota-\upsilon) e^{-\delta}}$, 
    for some positive function $C(\iota)$ depending only on the geometry of $M$.
\end{Claim}
\begin{proof}
    For the first inequality, the proof goes exactly as that of Claim~\ref{claim: inter_kernel}. For the smoothed version, consider the function $f(x)h_{\iota,\upsilon}(x)$ instead of $f(x)$. We have 
    $$|f(P)h_{\iota,\upsilon}(P)-f(v)h_{\iota,\upsilon}(v)|\le |h_{\iota,\upsilon}(P)||f(P)-f(v)|+|f(v)||h_{\iota,\upsilon}(P)-h_{\iota,\upsilon}(v)|.$$
     Adding up the first term above gives the error bound $\delta\Lip(f)$. The second term above is only nonzero for intersection point $P$ in the region $\mathcal{R}$. Denote by $I_\mathcal{R}(\Gamma,S)$ the sum of delta measures of intersection points in $\mathcal{R}$. So adding up the second term above gives the bound
     $$\frac{|I_{\mathcal{R}}(\Gamma,S)|}{|I_{\iota,\upsilon}(\Gamma,S)|}\Lip(h_{\iota,\upsilon})\cdot\delta\cdot\sup_{\mathcal{R}}|f|$$
     We can then apply Lemma~\ref{lem: ratio}, and note that $\Lip(h_{\iota,\upsilon})=O(\iota/\upsilon)$. This gives the desired estimate.
\end{proof}

\begin{constr}
    We need $\delta\to0$ as $\vol(\Gamma),\vol(S)\to\infty$. We also want $\delta/\upsilon\to0$.
\end{constr}

\paragraph{Step 2: Smooth the kernel}
The key technical part, as before, is to smooth the kernel. We again describe the construction in $\mathbb{H}^d$ first, and then map the local construction isometrically to $M$.

Fix $o\in\mathbb{H}^d$ and $v_0\in\mathcal{F}_o(\mathbb{H}^d)$. Let $\rho<\delta$ be a positive parameter to be determined later. Let $\mathcal{N}_{\delta,\rho}(v_0)$ be the $\rho$-neighborhood of $B_{2\delta}^{d-k}(v_0)$.

Suppose $t<\rho$ is another parameter to be determined later. As before, let $\chi_\delta^{t,\rho}$ be a smooth function identically $1$ outside $\mathcal{N}_{\delta,\rho}(v_0)$, and identically $0$ on $\mathcal{N}_{\delta,\rho-t}(v_0)$. In the previous section, if the geodesic segment $B^1_\delta(u)$ intersects $B^{d-1}_\delta(v_0)$ non-transversely, then $\pi(u)$ must be contained in $B_{2\delta}^{d-1}(v_0)$. However, this is no longer true when $k\ge2$. So only taking away the frames over the neighborhood $\mathcal{N}_{\delta,\rho}(v_0)$ is not enough to resolve non-smoothness / discontinuity.


Let $\mathcal{P}^{d-k}(v_0)$ be the $(d-k)$-dimensional totally geodesic subspace tangent to $v_0$ (cf.\ \S \ref{sec: inter_kernel}). For totally geodesic hypersurfaces, we extend the definition in \ref{sec: inter_kernel} as follows. For $i=1,\ldots,d$, let $\mathcal{P}^{d-1}_i(v)$ be the totally geodesic hypersurface tangent to all the vectors except the $i$-th one in the frame $v\in\mathcal{F}_P(\mathbb{H}^d)$. In particular $\mathcal{P}^{d-1}_d(v)=\mathcal{P}^{d-1}(v)$.

Given $P\notin\mathcal{P}^{d-k}(v_0)$, we now define a natural subset to exclude non-transverse intersections. Let $\mathcal{J}_i(P)$ be the set of frames $u\in\mathcal{F}_P(\mathbb{H}^d)$ based at $P$ so that $\mathcal{P}^{d-1}_i(u)$ contains $\mathcal{P}^{d-k}(v_0)$, and set $\mathcal{J}(P):=\bigcup_{i=k+1}^d\mathcal{J}_i(P)$. In particular, if $B_\delta^k(u)$ and $B_\delta^{d-k}(v_0)$ intersect non-transversely then $u\in\mathcal{J}(P)$ for a point $P$ within $\delta$-distance of $B_\delta^{d-k}(v_0)$.

Alternatively, $\mathcal{J}(P)$ may be constructed as follows. There exists a unique $(d-k+1)$-dimensional totally geodesic subspace $\mathcal{Q}$ containing $\mathcal{P}^{d-k}(v_0)$ and $P$, and $\mathcal{J}_i(P)$ consists of all frames $u$ at $P$ so that $\mathcal{P}^{d-1}_i(u)$ contains $\mathcal{Q}$. Hence each component of $\mathcal{J}(P)$ is a submanifold of $\mathcal{F}_P(\mathbb{H}^d)$ corresponding to different ways to extend a frame over $\mathcal{Q}$ to a full frame. Moreover, each connected component has codimension $(k-1)(d-k+1)$ in $\mathcal{F}_P(\mathbb{H}^d)$.

Note that we can put a smooth metric on $\mathcal{F}(\mathbb{H}^d)$ invariant under isometries. Because of this, we can choose a uniform small enough $\rho>0$ independent of $P$ so that $\rho$-neighborhoods of the connected components of $\mathcal{J}(P)$ in $\mathcal{F}_P(\mathbb{H}^d)$ are disjoint from each other. Let $\zeta_P$ be a smooth function identically $1$ outside these $\rho$-neighborhoods, and identically $0$ on $(\rho-t)$-neighborhoods of the connected components. Define the following function $\zeta^{t,\rho}(u)=\zeta_{\pi(u)}(u)$. It is easy to see that this function is smooth when $\pi(u)\notin\mathcal{P}^{d-k}(v_0)$. Moreover, we can arrange so that its $m$-th derivatives have size $O(\rho^{-m}t^{-m})$. 

Now we can transport these functions to (frames over) neighborhoods defined for $v\in\mathcal{F}(M)$, which we denote by $\chi_{v,\delta}^{t,\rho}$ and $\zeta_v^{t,\rho}$. Moreover, we may assume the $m$-th derivatives of $\chi_{v,\delta}^{t,\rho}$ have size $O(t^{-m})$, and those of $\zeta_v^{t,\rho}$ have size $O(\rho^{-m}t^{-m})$.

Finally, recall that $\eta_\delta^{t}:\mathbb{R}\to\mathbb{R}$ is a smooth even function supported on the interval $[-\delta,\delta]$ and identically $1$ in $[-\delta+t,\delta-t]$. We also assume that $\eta_\delta^{t}$ is monotone on $\mathbb{R}_+$ and have $m$-th derivatives of size $O(t^{-m})$. Define
\begin{align*}
\mathbb{K}_{\delta}^{t,\rho}(u,v)&=
\begin{cases}
    \frac{1}{c_\delta}\chi_{v,\delta}^{t,\rho}(\pi(u))\cdot\zeta_v^{t,\rho}(u)\cdot\eta_\delta^{t}(d(P,\pi(v)))&\text{if $B^k_\delta(u)\pitchfork B^{d-k}_{\delta}(v)$ at $P$},\\
    0 &\text{otherwise.}
\end{cases}\\
\widetilde{\mathbb K}_{\delta}^{t,\rho}(u,v)&=
\begin{cases}
    \mathbb{K}_{\delta}^{t,\rho}(u,v)\cdot\eta^{t}_\delta(d(P,\pi(u)))&\text{if $B^k_\delta(u)\pitchfork B^{d-k}_{\delta}(v)$ at $P$},\\
    0 &\text{otherwise.}
\end{cases}
\end{align*}

As a guide for readers,
\begin{itemize}[topsep=0mm, itemsep=0mm]
    \item $\chi$ is used to take out a $\rho$-neighborhood near $B_{2\delta}^{d-k}(v)$ smoothly;
    \item $\zeta$ is used to smooth the discontinuity between frames giving transverse and non-transverse intersections; and
    \item $\eta$ is used to smooth the kernel near the boundary of the region, when the intersection point $P$ is close to the boundaries of $B_{\delta}^{k}(u)$ or $B_{\delta}^{d-k}(v)$.
\end{itemize}

We can then define the associated kernels $K_\delta^{t,\rho}(u,v)$ and $\widetilde{K}_\delta^{t,\rho}(u,v)$ on $\mathcal{F}(M)$. These kernels defined here have similar properties as those defined in the case $k=1$.
\begin{Lemma}\label{lem:kernel_smooth_prop_gen}
\begin{enumerate}[label=\normalfont{(\arabic*)}]
    \item For a fixed $v$, $\mathbb{K}_\delta^{t,\rho}(u,v)$ differs from $\mathbb{K}_\delta(u,v)$ over a region of volume $O(\delta^{d-1}\rho)$ if $t=O(\rho^2)$ and $\delta=O(1)$.
    \item $\widetilde{\mathbb K}_\delta^{t,\rho}(u,v)$ is smooth. Consequently, $\widetilde{K}_\delta^{t,\rho}(u,v)$ is smooth.
    \item The $m$-th derivatives of $\widetilde{\mathbb K}_\delta^{t,\rho}$ have size $O(\delta^{-d}\rho^{-m}t^{-m})$. Consequently, over $M_{\ge\iota}$, the $m$-th derivatives of $\widetilde{K}_\delta^{t,\rho}(u,v)$ have size $O(\lceil\delta/\iota\rceil^{d-1}\delta^{-d}\rho^{-m}t^{-m})$.
    \item For a fixed $v$, $\widetilde{\mathbb{K}}_\delta^{t,\rho}(u,v)$ is nonzero on a region with volume $\asymp \delta^d$ if $\rho/\delta \to 0$ as $\delta \to 0$.
\end{enumerate}
\end{Lemma}
\begin{proof}
    We only explain the proof of Part (1); other parts can be proved similarly as Lemma~\ref{lem:kernel_smooth_prop}. Compared to the case $k=1$, now we have an additional region to account for, which consists of frames $u$ in $\mathcal{J}_P$ for some $P$ within $\delta$-neighborhood of $B^{d-k}_{\delta}(v)$. It is easy to see that this has volume $O(\delta^d\rho^{(k-1)(d-k+1)})=O(\delta^{d-1}\rho)$.
\end{proof}
We thus have the following claim, which can be proved analogously as Claim~\ref{claim: smoothing}.
\begin{Claim}\label{claim: smoothing_gen}
Suppose $t=O(\rho^2)$ and $\delta=O(1)$. We have
    \begin{align}
        &\left|\int_{\mathcal{F}(M)} f(v) \int_{\mathcal{F}(M)} \widetilde{K}^{t,\rho}_{\delta}(u,v) \, d\mu^1_{\Gamma}(u)d\mu^1_{S,\iota,\upsilon}(v)-\int_{\mathcal{F}(M)} f(v) \int_{\mathcal{F}(M)} K_{\delta}(u,v) \, d\mu^1_{\Gamma}(u)d\mu^1_{S,\iota,\upsilon}(v)\right| \nonumber \\ 
        &\leq C\lceil\delta/\iota\rceil^{d-1}\sup|f|\left(\frac{t}{\delta^{d+1}}+\frac{\rho^q}{t^q\delta^d}\vol(\Gamma)^{-\epsilon_2}\right)+C\sup |f|\frac{\rho}{\delta \vol(M)}. 
        \end{align}
where $q,\epsilon_2$ are the constants from Theorem~\ref{thm: mixing.surface}.
\end{Claim}
\begin{constr}\label{constr: smoothing_gen}
    We need $t=O(\rho^2)$, $\delta=O(1)$, $t/\delta^{d+1}\to0$, $\rho/\delta\to0$, and $\frac{\rho^q}{t^q\delta^d}\vol(\Gamma)^{-\epsilon_2}\to0$.
\end{constr}

\paragraph{Step 3: Apply the equidistribution results}
We can now apply the two equidistribution results to get integration over the entire frame bundle.
\begin{Claim}\label{claim: equi_geod_gen}
Replacing $\mu_\Gamma^1$ with $L_M^1$, we have
\begin{align}
        &\left|\int_{\mathcal{F}(M)} f(v) \int_{\mathcal{F}(M)} \widetilde{K}^{t,\rho}_{\delta}(u,v) \, d\mu^1_{\Gamma}(u)d\mu^1_{S,\iota,\upsilon}(v)-\int_{\mathcal{F}(M)} f(v) \int_{\mathcal{F}(M)} \widetilde{K}^{t,\rho}_{\delta}(u,v) \, dL^1_M(u)d\mu^1_{S,\iota,\upsilon}(v)\right| \nonumber \\
        &\leq C\lceil\delta/\iota\rceil^{d-1}\delta^{-d}\rho^{-q}t^{-q}\vol(\Gamma)^{-\epsilon_2}\sup|f|.
        \end{align}

\end{Claim}
\begin{constr}\label{constr: equi_geod_gen}
    We need $\delta^{-d}\rho^{-q}t^{-q}\vol(\Gamma)^{-\epsilon_2}\to0$.
\end{constr}

Up to this point, we allow $k=d-1$, i.e. $S$ may be closed geodesics. For the next claim, assume $k<d-1$ as well.
\begin{Claim}\label{claim: equi_hyper_gen}
Replacing $\mu_{S,\iota,\upsilon}^1$ with $L_{M,\iota,\upsilon}^1$ (where $L_{M,\iota,\upsilon}=h_{\iota,\upsilon}L_M$), if $\vol(S)^{\epsilon_2'}\upsilon^{q}\to\infty$, we have 
    \begin{align}
&\left|\int_{\mathcal{F}(M)} f(v) \int_{\mathcal{F}(M)} \widetilde{K}^{t,\rho}_{\delta}(u,v) \, dL^1_M(u)d\mu^1_{S,\iota,\upsilon}(v)-\int_{\mathcal{F}(M)} f(v) \int_{\mathcal{F}(M)} \widetilde{K}^{t,\rho}_{\delta}(u,v) \, dL^1_{M}(u)dL^1_{M,\iota,\upsilon}(v)\right|\nonumber \\  
&\leq C \iota^{-q}\vol(S)^{-\epsilon_2'} \upsilon^{-q}\Sob_q(f).
\end{align}
where $q,\epsilon_2'>0$ are the constants from Theorem~\ref{thm: mixing.surface}.
\end{Claim}
\begin{proof}
As before, the inner integral is now a constant $\asymp1$. So we only need to bound
$$\left|\int_{\mathcal{F}(M)}f\,d\mu^1_{S,\iota,\upsilon}-\int_{\mathcal{F}(M)}f\,dL^1_{M,\iota,\upsilon}\right|=\left|\frac{\vol(S)}{|\mu_{S,\iota,\upsilon}|}\int_{\mathcal{F}(M)}fh_{\iota,\upsilon}\,d\mu^1_{S}-\frac{\vol(M)}{|L_{M,\iota,\upsilon}|}\int_{\mathcal{F}(M)}fh_{\iota,\upsilon}\,dL^1_{M}\right|.$$
 By adding and subtracting the term
$$
\frac{\vol(S)}{|\mu_{S,\iota,\upsilon}|}\int_{\mathcal{F}(M)}fh_{\iota,\upsilon}\,dL^1_{M},
$$
using triangle inequality and Theorem~\ref{thm: mixing.surface}, we conclude that this is bounded by
$$C\frac{\vol(S)}{|\mu_{S,\iota,\upsilon}|}\vol(S)^{-\epsilon_2'}\Sob_q(fh_{\iota,\upsilon})+\left|\frac{\vol(S)}{|\mu_{S,\iota,\upsilon}|}-\frac{\vol(M)}{|L_{M,\iota,\upsilon}|}\right|\|fh_{\iota,\upsilon}\|_{L^1}.$$
Now $\Sob_q(fh_{\iota,\upsilon})=O(\Sob_q(f)\upsilon^{-q})$. Moreover, applying Theorem \ref{thm: mixing.surface} to the function $h_{\iota,\upsilon}$ we have
$$\frac{|\mu_{S,\iota,\upsilon}|}{\vol(S)}=\frac{|L_{M,\iota,\upsilon}|}{\vol(M)}+O(\vol(S)^{-\epsilon_2'}\Sob_q(h_{\iota,\upsilon}))=\frac{|L_{M,\iota,\upsilon}|}{\vol(M)}+O(\vol(S)^{-\epsilon_2'}\iota^{-q}\upsilon^{-q}).$$
The ratios $\vol(S)/|\mu_{S,\iota,\upsilon}|,\vol(M)/|L_{M,\iota,\upsilon}|$ are bounded above by a constant independent of $\iota$ and $\upsilon$, as the support of $h_{\iota,\upsilon}$ contains a definite thick part of $M$. Therefore, the last equation implies their difference goes to zero. Combining these with $\vol(S)^{\epsilon_2'}\upsilon^{q}\to\infty$, we have the desired upper bound.
\end{proof}
\begin{constr}\label{constr: equi_hyper_gen}
    We need $\vol(S)^{\epsilon_2'}\upsilon^{q}\to\infty$.
\end{constr}
\paragraph{Step 4: Bound the coefficient}
Finally, we have the bound on $c_{\iota,\upsilon}(\Gamma,S)$.
\begin{Claim}\label{claim: coeff_gen}
Let $A_2,A_3$ be the upper bounds we obtained in Claims~\ref{claim: smoothing_gen} and \ref{claim: equi_geod_gen} for function $f=1$. Then we have
$$
\left|1-c_{\iota,\upsilon}(\Gamma,S)\int_{\mathcal{F}(M)}\widetilde{K}^{t,\rho}_{\delta}(u,v)dL_M^1\right|\leq C\cdot \delta/\upsilon+c_{\iota,\upsilon}(\Gamma,S)(A_2+A_3).
$$
Consequently, if $A_2,A_3=o(1)$, then there exists a positive function $C(\iota)$ so that $c_{\iota,\upsilon}(\Gamma,S)\le C(\iota)$ as $\vol(\Gamma)\to\infty$.
\end{Claim}
\begin{constr}\label{constr: coeff_gen}
    Need $A_2,A_3=o(1)$, which are already covered by previous constraints.
\end{constr}

\paragraph{Step 5: Proof of the main theorem for $1<k<d-1$}
\begin{Lemma}\label{lem:constr_gen}
    For any $0<\kappa<\frac{\epsilon_2}{d+3qd/2+3q}$, set $\delta=\vol(\Gamma)^{-\kappa}$, $\rho=\vol(\Gamma)^{-(1+d/2)\kappa}$, $t=\vol(\Gamma)^{-(d+2)\kappa}$ and $\upsilon=\max\{\vol(\Gamma)^{-\kappa/2},\vol(S)^{-\epsilon'/(2q)}\}$, then all the constraints are satisfied.
\end{Lemma}

Using this lemma and arguing as in the previous section, we obtain the main theorem, with the following additional input to replace the unsmoothed restriction with the smoothed version.

\begin{Lemma}
    We have
    $$\left|\int_Mf\,dI_{\ge\iota}-\int_Mf\,dI_{\iota,\upsilon}\right|\le C(\iota)\cdot (\delta/\upsilon+A_2+A_3)\sup|f|.$$
\end{Lemma}
\begin{proof}[Sketch of proof]
    We only need to consider intersection points in a thin collar (width specified by $\upsilon$). We can apply the estimates above in this section to a function supported in a slightly larger neighborhood of the collar, while making sure its derivatives are bounded above by sup of those of $h_{\iota,\upsilon}$.
\end{proof}

\section{Joint equidistribution}\label{sec:joint}
In this part, we describe how to modify our argument to obtain the joint equidistribution result Theorem~\ref{thm: joint_equidistr}. For completeness, we include a description of the configuration space as a `nice' subspace of matrices.

\paragraph{Configuration space}
Let ${\Gr}(k,d)$ be the Grassmannian of $k$-planes in $\mathbb{R}^d$. Let $\Conf(k,d)$ denotes the set of all pairs of $k$-planes and $(d-k)$-planes in $\mathbb{R}^d$ up to the action of rigid rotations, which we call the \emph{configuration space}.

The configuration space $\Conf(k,d)$ can be described as follows. Given a $k$-plane and a $(d-k)$-plane in $\mathbb{R}^d$ (with standard basis $\{e_1,\ldots,e_d\}$), we can always apply rigid rotation so that the $(d-k)$-plane is spanned by $\{e_{k+1},\ldots,e_d\}$. The stabilizer of this standard $(d-k)$-plane in $\bigO(d)$ is $\bigO(k)\times \bigO(d-k)$, which acts on the $k$-plane as well. So we can view $\Conf(k,d)$ as the orbit space of $\Gr(k,d)$ under the action of $\bigO(k)\times\bigO(d-k)$. In particular, we can view any smooth function on $\Conf(k,d)$ as a smooth function on $\Gr(k,d)$ invariant under $\bigO(k)\times\bigO(d-k)$.

Note that $\Conf(k,d)\cong\Conf(d-k,d)$ naturally. Moreover, when $k\le d/2$, $\Conf(k,d)\cong\Conf(k,2k)$. Indeed, as above, any pair of a $k$-plane $\mathcal{P}_k$ and a $(d-k)$-plane $\mathcal{P}_{d-k}$ in $\mathbb{R}^d$ can always be rotated so that $\mathcal{P}_{d-k}=\spanset\{e_{k+1},\ldots,e_d\}$. We can then use a rotation in $\mathcal{P}_{d-k}$ to make sure the projection of $\mathcal{P}_k$ is contained in $\spanset\{e_{k+1},\ldots,e_{2k}\}$.

We now describe $\Conf(k,2k)$ as a subset of $2k\times k$ real matrices. 
As above, given two $k$-planes $\mathcal{P}$ and $\mathcal{P}'$, we will assume $\mathcal{P'}$ is spanned by the basis $\{e_{k+1},\ldots,e_{2k}\}$. Denote by $\pi_{\mathcal{P}'},\pi_{(\mathcal{P}')^{\perp}}$ the orthogonal projections onto $\mathcal{P}', (\mathcal{P}')^\perp$ respectively. For simplicity, given any vector $v$ in $\mathbb{R}^{2k}$, we denote $v':=\pi_{\mathcal{P}'}(v)$ and $v'':=\pi_{(\mathcal{P}')^\perp}(v)$.

We first find a specific orthonormal basis for $\mathcal{P}$. Suppose $v=v_1$ is a unit vector on $\mathcal{P}$ minimizing the length of $v'$
(if there are several choices, take any one).
Inductively, suppose $v_1,\ldots,v_s$ has been chosen. Then choose $v=v_{s+1}$ to be a unit vector in $\spanset\{v_1,\ldots,v_s\}^\perp\cap\mathcal{P}$ minimizing $v'$.
Note that as soon as $v_1,\ldots, v_{k-1}$ have been chosen, $v_k$ is determined uniquely by orthogonality up to sign.

Now we apply rotations in $\bigO(k)\times\bigO(k)$ to convert $\{v_1,\ldots,v_k\}$ into a particular form as follows. Note that this could change the $k$-plane $\mathcal{P}$, but it remains the same point in $\Conf(k,2k)$. When $k\ge 2$, by applying a rotation in $\bigO(d)$ that fixes $e_{1},\ldots,e_{k}$ to $\mathcal{P}$ (so we can see it as a rotation in $\spanset \{e_{k+1},\ldots,e_{2k}\}$), we may assume $v_1'\in\mathbb{R}_{\ge0}e_{k+1}$. Similarly, by applying a rotation that fixes $e_{k+1},\ldots,e_{2k}$, we may assume $v_1''\in\mathbb{R}_{\ge0}e_1$. Inductively, suppose that we have applied rotations (that change $\mathcal{P}$ and fix $\mathcal{P'}$) so that $\spanset\{v_1',\ldots,v_s'\}=\spanset\{e_{k+1},\ldots,e_{k+t}\}$ for some $t\le s$ and $\spanset\{v_1'',\ldots,v_s''\}=\spanset\{e_1,\ldots,e_l\}$ for some $l\le s$. 
When $t\le k-2$, we can use a rotation in $\spanset\{e_{k+t+1},\ldots,e_{2k}\}$ to make sure $v_{s+1}'\in\spanset\{e_{k+1},\ldots,e_{k+t+1}\}$, and $v_{s+1}\cdot e_{k+t+1}\ge0$. When $t=k-1$, we use a reflection to make sure $v_{s+1}\cdot e_k\ge0$. Similarly, we can arrange $v_{s+1}''\in\spanset\{e_1,\ldots,e_{l+1}\}$ and $v_{s+1}\cdot e_{l+1}\ge0$.

Let $A$ be the $2k\times k$ matrix $\begin{bmatrix}v_1&\cdots&v_k\end{bmatrix}=\begin{bmatrix}
    A''\\A'
\end{bmatrix}$, where $A''$ and $A'$ are $k\times k$ blocks. Note that $A$ satisfies the following conditions:
\begin{enumerate}[label=(\alph*)]
    \item $A^tA=I_n$;
    \item $A'$ and $A''$ are in row echelon form, and the first nonzero term (i.e.\ the pivot) on each nontrivial row is positive (here we do not assume the pivot of a row has to be $1$);
    \item The columns of $A''$ are $v_1',\ldots,v_k'$, and satisfy $\|v_1'\|\le \cdots\le\|v_k'\|$.
\end{enumerate}
We will denote the collection of all $2k\times k$ real matrices satisfying the three conditions above by $\mathcal{A}_k$.

Conversely, let $A$ be any matrix in $\mathcal{A}_k$, and $\mathcal{P}_A$ the $k$-plane spanned by the columns $v_1,\ldots,v_k$ of $A$. It is easy to see that the orthonormal basis $\{v_1,\ldots,v_k\}$ may be produced using the inductive procedure minimizing $\pi_{\mathcal{P}'}(v)$ described above. This gives a map $A\mapsto [\mathcal{P}_A]$ from $\mathcal{A}_k$ to $\Conf(k,2k)$, where $[\mathcal{P}_A]$ is the equivalence class of the pair $(\mathcal{P}_A,\mathcal{P}')$ in $\Conf(k,2k)$.

We claim the map $A\mapsto [\mathcal{P}_A]$ is a bijection. The procedure above producing a matrix $A$ from a given $k$-plane $\mathcal{P}$ implies that the map is surjective. To see injectivity, it suffices to show that the matrix produced from $\mathcal{P}$ is unique, i.e. independent of the choices made. Indeed, to choose $v_1$, we have to minimize $\|\pi_{\mathcal{P}'}(v)\|$ over all unit vectors $v$. The image of the unit sphere in $\mathcal{P}$ under $\pi_{\mathcal{P}'}$ is an ellipsoid of dimension $\le k$ in $\mathcal{P}'$. It is easy to see that either the vector minimizing $\|\pi_{\mathcal{P}'}(v)\|$ is unique up to sign, or such vectors form a lower dimensional sphere. In either case, there exists a rotation in $\mathcal{P}'$ sends any minimizing vector to another while fixing the ellipsoid. This implies that all the vectors chosen later are related by the same rotation, and hence together they give the same matrix $A$.

In this way, we give $\Conf(k,2k)$ a global system of coordinates, and it is easy to see that there exists a dense open subset with a smooth structure.



\paragraph{Intersection kernel}
We now define a kernel that takes into account how geodesic submanifolds intersect. Let $h$ be a smooth function on the configuration space $\Conf(k,d)$. Given $x\in M$, we can identify $T_xM$ equipped with the inner product induced from the hyperbolic metric isomorphically with $\mathbb{R}^d$ with the standard inner product. In this way, the corresponding configuration space of pairs of $k$-planes and $(d-k)$-planes in $T_xM$ can be naturally identified with $\Conf(k,d)$, and in fact does not depend on the choice of the identification $T_xM\cong\mathbb{R}^d$.


We consider only the compact case. In particular, we can define the kernel on the manifold instead of first going through the universal cover. We define the associated kernel as follows.
$$K^h_\delta(u,v)=\begin{cases}
    \frac{1}{c_\delta}h(\mathcal{P},\mathcal{Q})& \text{if }\overline{B_\delta^k(u)}\pitchfork \overline{B_\delta^{d-k}(v)}, \\
    0 &\text{otherwise,}
\end{cases}$$
where $\mathcal{P},\mathcal{Q}$ denote the tangent planes to $B_\delta^k(u)$ and $B_\delta^k(u)$, respectively, at their intersection point.

\paragraph{Sketch of the proof of Theorem \ref{thm: joint_equidistr}} We first have the following estimate, which we can prove similarly as Claim~\ref{claim: inter_kernel} and Claim~\ref{claim: inter_kernel_gen}. For any $f\in C^1(M)$ we have
    \begin{align*}
        \left|\int_{M\times\Conf(k,d)}f(x)h(\theta)\, dI^1_{conf}(\Gamma,S)-c(\Gamma,S)\int_{\mathcal{F}(M)}f(v)\int_{\mathcal{F}(M)}K^h_\delta(u,v)d\mu_\Gamma^1(u)d\mu_S^1(v)\right|\\\le\delta\Lip(f)
    \sup|h|.
    \end{align*}
    
Next we can smooth $K_\delta^h$ similarly as before and apply effective equidistribution to the inner integral, and obtain similar bounds as Claims~\ref{claim: smoothing_gen} and \ref{claim: equi_geod_gen} (with bounds for $h$ and its derivatives now also involved). For brevity, as the approach is analogous, we omit this and ignore the difference between $K_\delta^h$ and its smoothed version.

We then apply effective equidistribution to the outer integral and conclude
\begin{align*}
    \int_{\mathcal{F}(M)}f(v)\int_{\mathcal{F}(M)}K^h_\delta(u,v)\,dL^1_M\,d\mu_S^1(v)=\int_{\mathcal{F}(M)}f(v)\int_{\mathcal{F}(M)}K^h_\delta(u,v)\,dL^1_M(u)\,dL^1_M(v)\\+O(\Sob_q(f)\sup|h|\vol(S)^{-\epsilon_2'}).
\end{align*}
We again observe that $\int_{\mathcal{F}(M)}K^h_\delta(u,v)\,dL^1_M(u)$ is a constant independent of $v$. For simplicity, we write $\mathcal{L}_\delta(h):=\int_{\mathcal{F}(M)}K^h_\delta(u,v)\,dL^1_M(u)$. Note that when $\delta\ll1$, the local geometry is approximately Euclidean. We have
$$\mathcal{L}_\delta(h)=\int_{\mathcal{F}(M)}K^h_\delta(u,v)\,dL^1_M(u)=\frac1{c_\delta}\int_{\Gr(k,d)}\int_{\mathcal{F}_\theta}h(\theta)dL_\theta d\mu=\frac1{c_\delta}\int_{\Gr(k,d)}h(\theta)\vol(\mathcal{F}_\theta)d\mu,$$
where $\mu$ denotes the uniform measure on $\Gr(k,d)$, $\mathcal{F}_\theta$ denotes the collection of $k$-frames that meets $B_\delta^{d-k}(v)$ in a configuration represented by $\theta\in\Gr(k,d)$, and $L_\theta$ the induced measure on $\mathcal{F}_\theta$ from $L_M$. Here, we use the fact that $\delta\ll 1$ to ensure that the neighborhoods of the intersection points are disjoint and that the Grassmannian bundle over these local neighborhoods can be described as a product space.

When $\delta\ll 1$, $\vol(\mathcal{F}_\theta)\approx \delta^d V(\theta)$ plus higher order error terms, where $V(\theta)$ is a function depending only on $\theta$. Hence the limit
$$\mathcal{L}(h):=\lim_{\delta\to0}\mathcal{L}_\delta(h)$$
exists and is in fact a constant multiple of the integration of $h$ against $V(\theta)d\mu(\theta)$. Moreover, we can control
$|\mathcal{L}_\delta(h)-\mathcal{L}(h)|$ by sup norm of $h$ and some power of $\delta$. This gives the desired estimates.

\phantomsection
\addcontentsline{toc}{section}{References}
\bibliography{biblio.bib}
\bibliographystyle{math}

\end{document}